\documentclass[11pt,twoside, a4paper, english, reqno]{amsart}
\usepackage{graphicx, amsmath, varioref, amscd, amssymb,color, bm, stmaryrd, amsthm, epsfig, color, epic, enumerate}
\usepackage{fancybox, tcolorbox}
\usepackage{fancybox}
\usepackage[pdftex, colorlinks=true]{hyperref}
\usepackage{upref}
\usepackage{pdfsync}

\usepackage{eso-pic}
\usepackage{color}
\usepackage{type1cm}
\makeatletter
  \AddToShipoutPicture*{%
    \setlength{\@tempdimb}{.12\paperwidth}%
    \setlength{\@tempdimc}{.5\paperheight}%
    \setlength{\unitlength}{1pt}%
    \put(\strip@pt\@tempdimb,\strip@pt\@tempdimc){%
      \makebox(0,0){\rotatebox{90}{\textcolor[gray]{0.75}{\fontsize{2cm}{2cm}}}}
    }
}

\makeatother

\numberwithin{equation}{section}
\allowdisplaybreaks

\newtheorem{theorem}{Theorem}[section]
\newtheorem{lemma}[theorem]{Lemma}

\newtheorem{proposition}[theorem]{Proposition}
\theoremstyle{definition}
\newtheorem{definition}[theorem]{Definition}

\theoremstyle{remark}

\newcommand{\Div}{\operatorname{div}}

\newcommand{\Grad}{\nabla}

\newcommand{\vu}{\vc{u}}

\newcommand{\vc}[1]{{\bm{#1}}}

\newcommand{\R}{\mathbb{R}}
\newcommand{\N}{\mathbb{N}}

\newcommand{\e}{\varepsilon}

\newcommand{\eps}{\epsilon}

\begin{document}

\title[On a Free Boundary Problem for a Model of Polymeric Fluids]{On  a free boundary problem for finitely extensible bead-spring chain molecules in dilute polymers }

\author[Donatelli]{Donatella Donatelli}
\address[Donatelli]{\newline
Departement of Engineering Computer Science and Mathematics\\
University of L'Aquila\\
67100 L'Aquila, Italy.}
\email[]{\href{donatella.donatelli@univaq.it}{donatella.donatelli@univaq.it}}
\urladdr{\href{http://people.disim.univaq.it/~donatell/}{http://people.disim.univaq.it/\~{}donatell}}

\author[Trivisa]{Konstantina Trivisa}
\address[Trivisa]{\newline
Department of Mathematics \\ University of Maryland \\ College Park, MD 20742-4015, USA.}
\email[]{\href{http://www.math.umd.edu}{trivisa@math.umd.edu}}
\urladdr{\href{http://www.math.umd.edu/~trivisa}{math.umd.edu/\~{}trivisa}}

\date{\today}

\subjclass[2010]{Primary: 35Q30, 76N10; Secondary: 46E35.}

\keywords{FENE model; suspensions of  extensible bead-spring chain molecules; dilute polymenrs; compressible Navier-Stokes equations; Fokker-Planck-type equation, free boundary problems.}

\thanks{}

\maketitle

\begin{abstract}
We investigate the global existence of weak solutions to a free boundary problem governing the evolution of 
 finitely extensible bead-spring chains in dilute polymers. We construct weak solutions of  the two-phase  model by performing the asymptotic limit  as  the adiabatic exponent $\gamma$ goes to $\infty$  for  a macroscopic model which arises from the kinetic theory of dilute solutions of nonhomogeneous polymeric liquids. In this context the  polymeric molecules are idealized as bead-spring chains with finitely extensible nonlinear elastic  (FENE) type spring potentials. This class of models involves the unsteady, compressible, isentropic, isothermal Navier-Stokes system in a bounded domain $\Omega$ in $\R^d,$ $d=2, 3$ coupled with a  Fokker-Planck-Smoluchowski-type diffusion equation (cf.\ Barrett and S\"{u}li \cite{BS2011},\cite{BS2012A}, \cite{BS2016}). The convergence of these solutions, up to a subsequence, to the free-boundary problem is established using weak convergence methods, compactness arguments which rely on the monotonicity properties of certain quantities  in the spirit of \cite{DonatelliTrivisa2017}.
\end{abstract}


\section{Introduction}\label{S1}
Systems modeling the interaction of fluids and polymeric molecules  are of great scientific  interest in many branches of applied physics, chemistry, biology and engineering. They are of  use in many industrial and medical applications such as food processing and blood flows. Polymeric molecules are very complex objects, and their description and  investigation present many challenges.  One of the most interesting  models is the FENE (Finite Extensible Nonlinear Elastic) dumbbell model. In this model, a polymer is idealized as an �elastic dumbbell� consisting of �beads� joined by a spring. We refer the reader to Bird, Amstrong and Hassager \cite{BAH1977, BAH1987}, Doi and Edwards \cite{DoiEdwards1986} for some physical introduction to the model, \"{O}ttinger \cite{Ottinger1996} for a more mathematical treatment following the stochastic framework and Owens and Phillips \cite{OwensPhillips2002} for the computational aspects of the problem. 

In order to gain some perspective of the complexity of the problem let us recall that one of the starting points in the investigation of polymeric flows  is due to Kirkwood and Riseman, who treated the perturbation of the velocity field due to the polymer's presence  by steady state hydrodynamics ignoring the dynamical motion of the polymer. Subsequently, Bird, Curtis, Armstrong and Hassager in  \cite{BCAH1987} advanced significantly Kirkwood's early theory introducing a general kinetic theoretical framework for both diluted and concentrated polymeric systems. In that context, the macromolecules are modeled as freely jointed bead-rod or bead-spring chains. 

The configurational distribution function, solution of an evolution (diffusion) equation of the Fokker-Planck Smoluchowski-type, is the foundation of polymer dynamics: it is central to the estimation of the components of the stress tensor.
The behavior of the viscoelastic flow in polymeric liquids is effected significantly by the complexity of  inter- and intramolecular interactions. At the microscopic level, long chain entanglements are a consequence of chain connectivity and backbone uncrossability due to intermolecular repulsive exclusive volume forces. Macromolecules diffusion (and conformational relaxation) is slowed down due to hydrodynamic drag and Brownian forces.

The microscopic effect due to the interaction between the macroscopic compressible fluid and the polymeric bead-like molecules produces an extra stress term in the momentum equation. This effect is known as micro-macro interaction. Analogously, there is an extra  drift term in the Fokker-Planck equation that depends on the spatial gradient of the velocity. This term represents  a macro-micro effect. The coupling satisfies the fact that the free-energy dissipates, which is important not only  from the physical point of view but also from mathematical considerations, since it allows us to obtain uniform bounds and hence prove global existence of weak solutions.

The  resulting system offers a detailed description of  the behavior of the complex mixture of polymer molecules and compressible fluid, and as such, it presents numerous challenges, simultaneously at the level of their derivation [23], at the level of their numerical simulation, at the level of their physical properties (rheology) and that of their mathematical treatment (see references below).

This paper establishes the existence of global-in-time weak solutions  to a free boundary problem governing the evolution of 
 finitely extensible bead-spring chains in dilute polymers. The free boundary problem is  defined with the aid of a threshold for the pressure beyond which one has the incompressible Navier-Stokes equations for the fluid and below which one has a compressible model for the gas. We construct weak solutions of  the two-phase  model by performing the asymptotic limit  as  the adiabatic exponent $\gamma$ goes to $\infty$  for  a macroscopic model which arises from the kinetic theory of dilute solutions of nonhomogeneous polymeric liquids. In this context the  polymeric molecules are idealized as bead-spring chains with finitely extensible nonlinear elastic  (FENE) type spring potentials. This class of models involves the unsteady, compressible, isentropic, isothermal Navier-Stokes system in a bounded domain $\Omega$ in $\R^d,$ $d=2,$ or $3$ coupled with a  Fokker-Planck-Smoluchowski-type diffusion equation (cf. Barrett and S\"{u}li \cite{BS2011},\cite{BS2012A}, \cite{BS2016}). The convergence of these solutions, up to a subsequence, to the free-boundary problem  is established using techniques in the spirit of Lions and Masmoudi \cite{LionsMasmoudi-1999}. 
 
 For related work in the context of polymeric fluids we refer the reader to \cite{DonatelliTrivisa2017} where  the stability and global existence of weak solutions to a free boundary problem governing the evolution of  polymeric fluids is investigated.  The starting point  in the investigation  of Donatelli and Trivisa in \cite{DonatelliTrivisa2017} is a macroscopic model governing  the suspensions of rod-like molecules  (known as Doi-Model) in compressible fluids.  The model under consideration couples a Fokker-Planck-type equation on the sphere for the orientation distribution of the rods to the Navier-Stokes equations, which are now enhanced by additional stresses reflecting the orientation of the rods on the molecular level.  The coupled problem is 5- dimensional (three-dimensions in physical space and two degrees of freedom on the sphere) and it describes the interaction between the orientation of rod-like polymer molecules on the microscopic scale and the macroscopic properties of the fluid in which these molecules are contained. The macroscopic flow leads to a change of the orientation and, in the case of flexible particles, to a change in shape of the suspended microstructure. This process, in turn yields the production of a fluid stress.  The free-boundary problem is defined by a threshold for the pressure beyond which one has the incompressible Navier-Stokes equations for the fluid and below which one has a compressible model for the gas.
Regarding the literature on polymeric fluids for compressible flows we refer the reader to the articles by Bae and Trivisa\cite{BaeTrivisa2012, BT2013}, Barrett and S\"{u}li \cite{BS2012B, BS2012C}, Donatelli and Trivisa \cite{DonatelliTrivisa2017} and the reference therein.


\subsection{Notations:}
Before formulating the governing equation of the nonlinear system governing our mixture, we fix here some notations we are going to use in the paper.
\subsubsection{Notations of macroscopic variables,  tensors, forces  and coefficients}
\begin{enumerate}[\quad $\star$]
\item $\rho$ denotes the density of the fluid.
\item $\vc{u}$ represents the velocity field.
\item $\psi$ denotes the probability distribution function: $\psi = \psi(\vc{q})$ with $\vc{q}$ a random conformation vector of 
$\vc{q}= (\vc{q_1}^T, \dots, \vc{q_K}^T)^T \in  \R^{K_d}$ of the chain, with $\vc{q_i}$  representing the $d$-component conformation vector  of the $i$-th spring.
\item $(\rho, \vc{u}, \psi)$ denote the macroscopic variables which characterize the state of the polymeric fluid.
\item $M(\vc{q})$ denotes the total Maxwellian.
\item $\widehat{\psi}=\psi/M.$ 
\item $p= p(\rho)$ denotes the pressure.
\item $\mathbb{S}= \mathbb{S}[\rho, \vc{u}]$ denotes the viscous stress tensor.
\item $f$ denotes  a non-dimensional body force.
\item $\vc{\tau}$ denotes the elastic extra stress tensor:  $\vc{\tau} = \vc{\tau}(\psi).$
\item  $\vc{\sigma}(\vc{v}) = \Grad_x \vc{v}.$
\item $\zeta(\rho)$ denotes a drug coefficient $\zeta(\rho) \in \R, \,\, \zeta(\rho) >0.$
\item $D$ denotes the domain of admissible conformation vectors, $D \subset {\R}^K,$
$$D = D_1 \times \dots \times D_K,$$
$ D_i$  bounded open $d$-dimensional balls centered at the origin.
\item $\mathcal{O}_i:=\left[0, \frac{b_i}{2}\right)$ denotes the image of $D_i$ under $\vc{q_i} \in D_i \to \frac{1}{2} |\vc{q}_i|^2.$
\item $U_i$ denotes the spring potential, $U_i \in C^1(\mathcal{O_i}); \R_{\ge 0}), i=1, \dots, K.$  
\item $\vc{A} =(A_{i,j})_{i,j=1}^{K}$ is the symmetric positive definite {\it Rouse matrix} or connectivity matrix.
\item $\eta=\eta(x,t)$ denotes the polymeric number density expressed as 
$$\eta(x,t) = \int_D \psi(x,q,t) dq, \,\, (x,t) \in \Omega \times (0,T]$$
\item  $\mathcal{F}(s)=s(\log s-1)+1,$\qquad  $\displaystyle{P(s)=\frac{s^{\gamma}}{\gamma-1}}.$
\end{enumerate}
\subsubsection{Notations of function spaces}
\begin{enumerate}[\quad $\star$]
\item $L^{p}(0,T;X)$ denotes the Banach set of Bochner measurable functions $f$ from $(0,T)$ to $X$ endowed with either the norm $\Big(\int ^{T}_{0} \|g(\cdot, t)\|^{p}_{X}dt\Big)^{\frac{1}{p}}$ for $1\leq p<\infty$ or $\displaystyle\sup_{t>\infty} \|g(\cdot,t)\|_{X}$ for $p=\infty$. In particular, $f\in L^{r}(0, T; XY)$ denotes $\Big(\int ^{T}_{0} \big\|\big(\|f(t)\|_{Y_{\tau}}\big)\big\|^{p}_{X}dt\Big)^{\frac{1}{p}}$ or $\displaystyle\sup_{t>\infty} \big\|\big(\|f(t)\|_{Y_{\tau}}\big)\big\|_{X}$ for $p=\infty$. The notation
 $L^{p}_{t}L^{q}_{x}$ will abbreviate  the space $L^{p}(0,T;L^{q}(\Omega))$. 
 \item The space $L^{r}_{M}(\Omega\times D)$ denotes the space of measurable functions $f$ with the norm
 $\|f\|_{L^{r}_{M}(\Omega\times D)}=\Big(\int\!\!\int_{\Omega\times D}M|f|d\vc{q}dx\Big)^{1/r}$. For any $r\in[0,\infty)$ we define $Z_{r}=\{f\in L^{r}_{M}(\Omega\times D)\mid f\geq 0 \ \text{a.e. on}\  \Omega\times D\}$.
\item ${\mathcal M}((0,T) \times \Omega)$ is the space of bounded measures on $(0,T) \times \Omega$. 
\item $C(T)$ is a function only depending on initial data and $T$, $C_{w}([0,T];X)$, is the space of continuos function from $(0,T)$ to $X$ endowed with the weak topology.
\item $\rightharpoonup$ and $\rightarrow$ denote weak limit and strong limit, respectively.
\end{enumerate}

\subsection{Modeling}
 The main physical assumptions on our model are outlined below:
\begin{enumerate}[\quad $\star$]

\item
A macro-molecule is idealized as an ``elastic dumbbell" consisting of two
``beads" joined by a spring. The ``bead-spring chain model" (considered in the present article) 
consists of $K+1$ beads coupled with $K$ elastic springs representing a polymer chain.

\item
The polymer molecules are described by their density at each time $t$, position
$x$ and probability distribution $\psi$. This is a kinetic description of the polymer
molecules.

\item
The right-hand side of the Navier--Stokes momentum 
equation includes an elastic extra-stress tensor  $\vc{\tau}$ (produced due to the interaction of the compressible fluid and the polymeric molecules) which is the sum of the classical Kramers expression and a quadratic interaction term. The elastic extra-stress tensor stems from the random movement of the polymer chains and is defined through the associated probability density function that satisfies a Fokker--Planck-type parabolic equation, a crucial feature of which is the presence of a center-of-mass diffusion term.

\item 
The non-Newtonian elastic extra stress tensor $\vc{\tau}$  (cf. \eqref{1.5} below), depends on the probability density function $\psi,$ which, in addition to time  $t$  and space $x,$ also depends on the conformation vector 
$(q_1^T, . . . q_K^T )^T \in \mathbb{R}^{3K},$ with $q_i$ representing the $3$-component conformation/orientation vector of the $i$-th spring in the chain.

\item
 The Kolmogorov equation satisfied by $\psi$ is a second-order parabolic equation, the Fokker--Planck equation, whose transport coefficients depend on the velocity field $\vu$, and the hydrodynamic drag coefficient appearing in the Fokker--Planck equation is, generally, a nonlinear function of the density $\rho.$ 
\end{enumerate}

\subsection{Governing equations}
Our starting point is the governing equation of the general non-homogeneous bead-spring chain models with center of mass diffusion.
This class of models is governed by a system of nonlinear partial differential equations that arise from the kinetic theory of dilute polymeric solutions. The bead-spring chain molecules are dispersed in a compressible isentropic, isothermal Newtonian fluid confined to a bounded Lipschitz domain $\Omega \subset \R^d, d=2$ or $3,$ with boundary $\partial\Omega.$   The governing equation of the system {\bf (P)} read:
\begin{center}
\begin{tcolorbox}
\begin{equation} \label{1.1}
\partial_t \rho+\Div(\rho \vu)=0
 \end{equation}
\begin{equation}\label{1.2}
\partial_t (\rho  \vu) +\Div(\rho \vu\otimes \vu)+ \nabla_x p_{F}= \Div \mathbb{S}[\vc{u}, \rho] + \rho f + \Div \vc{\tau} 
\end{equation}
\begin{equation}\label{1.3}
 \partial_t \psi+\Div(\psi \vu) + \sum_{i=1}^K \nabla_{\vc{q_i}} \cdot (\vc{\sigma(u)} \vc{q_i}(\psi)) 
 \end{equation}
\[
= \eps  \Delta_{x}\left(\frac{\psi}{\zeta(\rho)}\right)+ \frac{1}{4\lambda} \sum_{i=1}^K \sum_{j=1}^K A_{i,j} \Grad_{\vc{q_i}} \cdot 
\left(M \Grad_{\vc{q_i}} \left(\frac{\psi}{\zeta(\rho)  M}\right)\right) 
\]
\end{tcolorbox}
\end{center}
The state of the mixture is characterized by the macroscopic variables:
 the density $\rho,$ the velocity vector field $\vu,$  and the probability distribution function $\vc{\psi}.$ 
 The physical properties of the polymeric fluid are reflected through the constitutive relations.These relations which are stated below  express how the fluid pressure $p_{F}$, the viscous stress tensor $\mathbb{S}= \mathbb{S}[\rho, \vc{u}],$ the elastic extra stress tensor $\vc{\tau}= \vc{\tau(\psi)}$ depend on the macroscopic variables and in addition specify the physical laws that characterize the behavior of the probability distribution function $\vc{\psi} = \vc{\psi(q)}$ in terms of the conformation vector $\vc{q}$ as well as other potentials present in the system. 
The constant parameter $\e$ denotes the center-of-mass diffusion coefficient, which is strictly positive.
The positive parameter $\lambda$ is called the {\em Deborah number}; it characterizes the elastic relaxation property of the fluid.  
\subsection{Constitutive relations} 
\begin{center}
\begin{enumerate}[\quad $\star$]
\item The viscous stress tensor $\mathbb{S}=\mathbb{S}[\rho, \vc{u}]$ follows  Newton's law for viscosity
\begin{equation} \label{1.5}
\mathbb{S}[\rho, \vu]= \mu^S(\rho) \left[D(\vc{u}) -\frac{1}{d} \Div \vc{u} \mathbb{I}
\right] + \mu^B(\rho) \Div \vc{u} \mathbb{I},
\end{equation}
where $\mu^S(\rho)>0,  \mu^B(\rho) \ge 0$ denote the shear and bulk viscosities, $\mathbb{I}$ the $d\times d$ identity tensor  
and $D(\vc{u})$ is the rate of strain tensor
$$D(\vc{u}):= \frac{1}{2} \left( \Grad_x \vc{u} + (\Grad_x \vc{u})^T \right),$$
with  $(\Grad_x \vc{u})(x,t) \in \R^{d\times d}$ and $(\Grad_x \vc{u})_{ij} =\displaystyle{\frac{\partial \vc{u}_i}{\partial x_j}}.$ 
\item  The elastic extra stress tensor $\vc{\tau}$ is defined by
\begin{equation}\label{1.6}
 \vc{\tau}(\psi)(x,t) := \vc{\tau_1}(\psi)(x,t) - \xi \left(\int_D  \psi d\vc{q}\right)^2 \mathbb{I}=\vc{\tau_1}(\psi)(x,t) - \xi \eta^2(x,t)\mathbb{I},
\end{equation} 
with $\xi>0$  and 
\begin{equation}\label{1.7}
\vc{\tau_1}(\psi):= k \left[\left(\sum_{i=1}^K \vc{C_i}(\psi)\right) - (K+1) \int_D \psi(\vc{q}) d\vc{q}\, \mathbb{I}\right], 
\end{equation} 
where $k >0$ and 

$$ \vc{C}_i(\psi)(x,t):= \int_d \psi(x, \vc{q},t) U_1'\left(\frac{1}{2} |\vc{q}_i|^2 \right)\vc{q} \vc{q_i}^T d\vc{q}.$$  
\end{enumerate}
\end{center}
\subsection{Additional hypothesis on the potentials} 
\begin{center}
\begin{enumerate}[\quad $\star$]
\item ${\vc{F}_i}$ denotes the elastic spring-force $\vc{F}_i: D_i \subset \R^d \to \R^d$ of the $i$-th   spring in the chain defined by 
\begin{equation}\label{1.8}
 \vc{F}_i (\vc{q_i}):=U'_i\left(\frac{1}{2} |\vc{q_i}|^2\right)\vc{q_i},\,\,  i=1, \dots, K.
 \end{equation}
Using this notation,  $\vc{C}_i(\cdot)$ can be expressed as
$$ \vc{C}_i(\psi)(x,t):= \int_d \psi(x, \vc{q},t)  \vc{F}_i (\vc{q_i}) \vc{q_i}^T d\vc{q}.$$  
\item $M_i$ represents the partial Maxwellian associated with the spring potential $U_i$ defined by 
\begin{equation}\label{1.9}
M_i(\vc{q_i}):= \frac{1}{\mathcal{Z}_i} e^{-U_i(\frac{1}{2}|\vc{q_i}|^2)}, \,\, \mathcal{Z}_i:= \int_{D_i} e^{-U_i(\frac{1}{2}|\vc{q_i}|^2)} d\vc{q_i}. 
\end{equation}
The (total) Maxwellian in the model is then 
\begin{equation}\label{1.10}
M(\vc{q}):= \displaystyle{{\prod}_{i=1}^K M_i(\vc{q_i})}\,\,\, \forall \vc{q} \in D.
 \end{equation}
\item Observe that,
\begin{equation}\label{1.11}
M(\vc{q}) \Grad_{\vc{q_i}} [M(\vc{q})]^{-1} = -[M(\vc{q})]^{-1}  \Grad_{\vc{q_i}} [M(\vc{q})]
\end{equation}
\[= \Grad_{\vc{q_i}} \left(U_i \left(\frac{1}{2}|\vc{q_i}|^2\right)\right) = U_i'\left(\frac{1}{2} |\vc{q_i}|^2\right)\vc{q_i},\] 
and by definition
$$ \int_D M(\vc{q}) d \vc{q} =1.$$
\item  We assume that $D_i = B(0, \sqrt{b_i})$ with $b_i >0$ is the ball centered at the origin of radius $\sqrt{b_i}$ and that  there exist constants $c_{ij} >0 \, j=1,\dots, 4,$ and 
$\theta_i > 1$ such that the spring potential $U_i \in C^1[0, \frac{b_i}{2})$ and the associated partial Maxwellian  $M_i$ satisfy $\forall \vc{q_i} \in D_i$
\begin{equation}\label{1.12}
c_{i1}[\mbox{dist}(\vc{q_i}, \partial D_i)]^{\theta_i} \le M_i(\vc{q_i}) \le ci2[\mbox{dist}(\vc{q_i}, \partial D_i)]^{\theta_i}
\end{equation}
\begin{equation}\label{1.13}
c_{i3} \le \mbox{dist}(\vc{q_i}, \partial D_i) U_i'\left( \frac{1}{2}|\vc{q_i}|^2\right)\le ci4.
\end{equation}
It follows from \eqref{1.13} that
\begin{equation}\label{1.14a}
\int_{D_i} \left[1 + \left[ U_i \left(\frac{1}{2}|\vc{q_i}|^2\right)\right]^2\right] M_i(\vc{q_i}) d\vc{q_i} < \infty.
\end{equation}
\end{enumerate}
\end{center}
From now on, for simplicity, we will assume that the viscosity coefficients $\mu^{S}, \mu^{B}$ will not depend on the density $\rho$ and we set the drag coefficient $\zeta=1$.

\subsubsection{Further Examples}
In the classical FENE dumbbell model, $K = 1$ and the spring is given by 
$$ \vc{F}(\vc{q})) = \left(1-\frac{|\vc{q}|^2}{b}\right)^{-1} \vc{q}, \vc{q} \in D = B(0,b^{\frac{1}{2}}),$$ corresponding to

$$U (s) = -\frac{b}{2}\log\left(1-\frac{2s}{b}\right) \in O = [0, b ), b > 2.$$
 More generally, in the FENE bead-spring 
 chain model, one considers $K +1$ beads linearly coupled with $K$ springs, each
with a FENE spring potential. Direct calculations show that the partial Maxwellian
$M_i$ and the elastic potentials $U_i, i = 1, . . . , K,$ of the FENE bead spring chain satisfy
the conditions \eqref{1.12} and \eqref{1.13} with $\theta_i := \frac{b_i}{2},$ provided that $b_i > 2, i = 1,...,K.$
Thus, \eqref{1.14a}  also holds when $b_i > 2, i = 1,...,K.$

\subsection{The free boundary problem}
In this article we are concerned with the free-boundary problem for the system \eqref{1.1}- \eqref{1.11}. In particular we consider a compressible model which includes the incompressible case only beyond a  certain threshold for the pressure. Indeed the model  the free boundary problem {($\vc{P_F}$)}  is defined by the following  equations in $(0,T)\times \Omega$:
\begin{center}
\begin{tcolorbox}
\begin{equation} \label{1.14}
\partial_t \rho+\Div(\rho \vu)=0
 \end{equation}
\[
\quad  0 \le \rho \le 1 \nonumber
\]
\begin{equation}\label{1.15}
\partial_t (\rho  \vu) +\Div(\rho \vu\otimes \vu)+ \nabla_x [p_{F}+ \xi \eta^2]= \Div \mathbb{S}[\vc{u}, \rho] + \rho f + \Div \vc{\tau_{1}} 
\end{equation}
\begin{equation}\label{1.16}
 \partial_t \psi+\Div(\psi \vu) + \sum_{i=1}^K \nabla_{\vc{q_i}} \cdot (\vc{\sigma(u)} \vc{q_i}(\psi)) 
 \end{equation}
\[
= \eps  \Delta_{x}\left(\frac{\psi}{\zeta(\rho)}\right)\frac{1}{4\lambda} \sum_{i=1}^K \sum_{j=1}^K A_{i,j} \Grad_{\vc{q_i}} \cdot 
\left(M \Grad_{\vc{q_i}} \left(\frac{\psi}{\zeta M}\right)\right) 
\]
\begin{equation}\label{1.17}
 \frac{\partial \eta}{\partial t} + \Grad_x \cdot (\vc{u} \eta) = \e \Delta_x \left(\frac{\eta}{\zeta(\rho)}\right) 
 \end{equation}
and the  free boundary conditions
\begin{equation}\label{1.18}
\Div \vu = 0 \quad \mbox{a.e} \,\, \mbox{on} \,\, \{ \rho = 1\}
\end{equation}
\begin{equation}\label{1.19}
p_{F} \ge  0  \quad \mbox{a.e} \,\, \mbox{in} \,\, \{ \rho = 1\}
\end{equation}
\begin{equation}\label{1.20}
p_{F} =  0  \quad \mbox{a.e} \,\, \mbox{in} \,\, \{ \rho < 1\}
\end{equation}
\end{tcolorbox}
\end{center}

The unknowns for our problem are the density $\rho,$ the velocity vector field $\vu,$ the pressure $p_{F}$, which is Lagrange multiplier associated with the incompressibility constraint \eqref{1.18}  $\Div \vu = 0$ a.e. in $\{\rho = 1\},$ the probability distribution $\psi$ and the polymeric number density $\eta$. It is important to observe that the pressure $p_{F}$ appears only in what we call the congested regions $\{\rho =1\}$ and that the conditions \eqref{1.19}, \eqref{1.20},  can be formulated as one constraint
\begin{equation}
\rho p_{F} = p_{F}\ge 0.
\label{cp}
\end{equation}
\\
\subsubsection{Boundary conditions}

Let $\partial \bar{D}_i:= D_1 \times \dots \times D_{i-1}\times \partial D_i \times D_{i+1}\times \dots \times D_K.$ 
We consider the problem {$\mathbf{(P_F)}$}  on a bounded domain with the following boundary conditions.

 \begin{eqnarray}
 && \vu=0 \,\,  \, \mbox{ on} \,\  \partial \Omega. \label{1.21}\\
&& \left[\frac{1}{4 \lambda}\sum_{j=1}^K A_{ij} M\Grad_{\vc{q_j}}\left(\frac{\psi}{\zeta(\rho)M}\right) - \sigma(\vc{u})\vc{q_i} \psi\right]\cdot \frac{\vc{q_i}}{|\vc{q_i}|} =0,\label{1.22}\\
&& \quad \quad \quad   \mbox{on}\,\,  \Omega \times \partial \bar{D}_i \times (0, T], \,\,\mbox{for}\,\, i=1, \dots, K,  \nonumber \\
&& \e \Grad_x \left(\frac{\psi}{\zeta(\rho)}\right) \cdot \vc{n} =0 \,\, \mbox{on}\,\, \partial \Omega \times D \times (0,T].\label{1.23}
\end{eqnarray}
where $\vc{q_j}$ is normal to $\partial D_i,$ as $D_i$ is a bounded ball centered at the origin, and $\vc{n}$ is normal to $\partial \Omega.$ 
\subsubsection{Initial data}
The system {$\mathbf{(P_F)}$} must be complemented with initial conditions, namely
\begin{eqnarray}
& & \rho(x,0) = \rho_0(x),\,\, \rho \vc{u} (x,0) = (\rho_0 \vc{u}_0)(x),\,\,  x \in \Omega,  \label{1.24}\\
& & \psi(\cdot, \cdot, 0) = \psi_0(\cdot, \cdot)\ge 0,\,\, \mbox{on} \,\,  \Omega \times D.\label{1.25}\\
& & \eta(x,0) = \int_D \psi_0(x, \vc{q}) d \vc{q}, \,\, \mbox{for}\,\, x\in \Omega.\label{1.26}
\end{eqnarray}

\subsection{Outline and overall strategy of the proof}
The outline of this article and overall strategy of the proof are as follows:
Section \ref{S1} presents  the main motivation for the upcoming investigation, the modeling aspects of the problem: the physical setting, constitutive relations,  the free-boundary problem, the statement of the problem and. Section \ref{S2} introduces  the main result, namely the global existence of the weak solutions to the free-boundary problem.  This is achieved by  rigorously  showing that these solutions can be obtained as the limit of weak solutions to the Doi model for compressible fluids as the adiabatic exponent $\gamma_n \to \infty.$ The approximating scheme and  an outline of the  proof of the global existence of approximate solutions are presented in Section \ref{S3}. The proof follows the line of argument  introduced by Barrett and S\"{u}li   \cite{BS2016} which is based on the  use of the rather special quantity 
$$\frac{\psi}{M}, \,\,\,\,\, M \, \mbox{the Maxwellian},$$
which if bounded appropriately  results to a new formulation of the equation verified by the probability distribution function  and an additional  partial differential equation. Section \ref{S4} presents the proof of the main theorem. 
The global existence of weak solutions to the free boundary problem is obtained by (a) showing the convergence of 
$(\rho_n -1)_+ \to 0$; (b) establishing the $L^1$ uniform bound for the approximate pressure $\rho_n^\gamma$;
(c) establishing the convergence of the approximating sequence $(\rho_n, \vu_n, \psi_n, \eta_n)$  through the proof of compactness for the solution sequence  by  using monotonicity properties of certain crucial quantities that depend on the macroscopic variables. 
Section \ref{S5} presents further extensions and related models.


\section{Main Result}\label{S2}
The goal of this paper is to prove the  existence of weak solutions to the free-boundary problem \eqref{1.15}-\eqref{1.20}, so we introduce the notion of weak solutions we are going to use throughout the paper. 
\subsection{Notion of weak solution to problem  $(P_F)$ } 
 
 \begin{definition} \label{def1}
{\bf [Weak solution of the problem  {($\vc{P_F}$)} ]}\\
A vector $(\rho, \vu, \psi, \eta)$ is called a weak solution to \eqref{1.14}-\eqref{1.20} with boundary data \eqref{1.21}-\eqref{1.23} and  initial data \eqref{1.24}-\eqref{1.26} if the equations
\begin{equation} \label{2.1}
\partial_t \rho+\Div(\rho \vu)=0
 \end{equation}
\begin{equation}\label{2.2}
\partial_t (\rho  \vu) +\Div(\rho \vu\otimes \vu)+ \nabla_x (p_{F}+\xi\eta^{2})= \Div \mathbb{S}[\vc{u}, \rho] + \rho f + \Div \vc{\tau_{1}} 
\end{equation}
\begin{equation}\label{2.3}
 \partial_t \psi+\Div(\psi \vu) + \sum_{i=1}^K \nabla_{\vc{q_i}} \cdot (\vc{\sigma(u)} \vc{q_i}(\psi)) 
 \end{equation}
\[
= \eps  \Delta_{x}\left(\frac{\psi}{\zeta(\rho)}\right)\frac{1}{4\lambda} \sum_{i=1}^K \sum_{j=1}^K A_{i,j} \Grad_{\vc{q_i}} \cdot 
\left(M \Grad_{\vc{q_i}} \left(\frac{\psi}{\zeta(\rho)  M}\right)\right) 
\]

\begin{equation}\label{2.4}
\partial_t \eta + \Div(\eta \vu)   - \e \Delta_x \left(\frac{\eta}{\zeta(\rho)}\right) =0
\end{equation}
are satisfied in the sense of distributions, the divergence free condition $\Div \vu = 0$ is satisfied a.e. in   $\{ \rho = 1\},$
the constrained $0\le \rho \le 1$ is satisfied a.e. in $(0,T) \times \Omega$ and the following regularity properties hold
\[\rho \in C([0,T]; L^p(\Omega)), \,\, 1\le p < \infty,\]
\[\vu \in L^2(0,T;(W_0^{1,2}(\Omega))),\,\, \rho|{\bf u}|^2 \in L^{\infty}(0, T;L^1(\Omega)),\]
\[p_{F} \in {\mathcal M}((0,T) \times \Omega)\]
\[\eta \in L^{\infty}(0,T; L^{2}(\Omega))\cap L^{2}(0,T; \dot{H}^{1}(\Omega))\]
\[\widehat{\psi}\in L^{p}(0,T;Z_{1})\cap H^{1}(0,T;M^{-1}(H^{s}(\Omega\times D))'),  \,\, 1\le p < \infty.\]
\end{definition}
Moreover $p_{F} $ is so regular that the condition 
\[p_{F}(\rho-1)=0, \]
is satisfied in the sense of distribution.
In this work we will  prove the existence of weak solutions to the free-boundary problem \eqref{1.14}-\eqref{1.20} by showing rigorously that they   can be obtained as a limit of  the weak solutions to the macroscopic fluid-particle problem
\begin{equation} \label{2.5}
\partial_t \rho_n+\Div(\rho_n \vu_n)=0
 \end{equation}
\begin{equation}\label{2.6}
\partial_t (\rho_n  \vu_n) +\Div(\rho_n \vu_n \otimes \vu_n)+ \nabla_x (p_{Fn}+\xi\eta_{n}^{2})= \Div \mathbb{S}[\vc{u_n}, \rho_n] + \rho f + \Div \vc{\tau_{1n}}
\end{equation}
\begin{equation}\label{2.7}
 \partial_t \psi_n+\Div(\psi_n \vu_n) + \sum_{i=1}^K \nabla_{\vc{q_i}} \cdot (\vc{\sigma(u_n)} \vc{q_i}(\psi_n)) 
 \end{equation}
\[
= \eps  \Delta_{x}\left(\frac{\psi_n}{\zeta(\rho_n)}\right)\frac{1}{4\lambda} \sum_{i=1}^K \sum_{j=1}^K A_{i,j} \Grad_{\vc{q_i}} \cdot 
\left(M \Grad_{\vc{q_i}} \left(\frac{\psi_n}{\zeta(\rho_n)  M}\right)\right) 
\]

\begin{equation}
\partial_t \eta_n+\nabla\cdot(\eta_n\vu_n) - \e \Delta_x \left(\frac{\eta_n}{\zeta(\rho_n)}\right) =0, \label{2.8}
\end{equation}
where $\vc{\tau_{1n}}=\vc{\tau_{1}}(\psi_{n})$ and 
$$p_{Fn}=(\rho_{n})^{\gamma_{n}},\  \gamma_{n}\to \infty,\ \text{as $n\to \infty$}.$$

Now we are ready to state the main existence results for our problem.
\begin{theorem}
Assume that the boundary conditions \eqref{1.21}-\eqref{1.23} and the initial conditions \eqref{1.24}-\eqref{1.26} are satisfied. Then,  there exists a weak solution (in the sense of Definition \ref{def1}) of the problem \eqref{1.14}-\eqref{1.20}. 
\label{MT}
\end{theorem}
The main Theorem \ref{MT} will be obtained  as  a consequence of the following stability result.

\begin{theorem} 
Let $n \in \N$ be fixed, then there exists a global weak solution  $(\rho_n, \vu_n, \psi_n, \eta_n)$ to 
\eqref{2.5}-\eqref{2.8} such that, as $n \to \infty$
\begin{equation}
\label{2.1.1}
(\rho_{n}-1)_{+}\rightarrow 0\qquad \text{in $L^{\infty}(0,T;L^{p})$, for any $1\leq p<+\infty.$}
\end{equation}
Moreover,
\begin{equation}
(\rho_{n})^{\gamma_{n}} \qquad \text{is bounded in $L^{1}$, for $n$ such that $\gamma_{n}\geq 3$,}
\end{equation}
and up to a subsequence there exists $p_{F}\in \mathcal{M}((0,T)\times \Omega)$ such that
\begin{equation}
(\rho_{n})^{\gamma_{n}} \rightharpoonup p_{F} \qquad \text{as $n\to \infty$}.
\end{equation}
If in addition $\rho_{n0}=\rho_{n}(x,0)\to\rho_{0}$ in $L^{1}$, then the following convergence holds:
 \[
 \rho_n \rightharpoonup \rho \,\, \mbox{weakly in} \,\, L^p((0,T) \times \Omega) \,\, 1\le p < +\infty,
 \]
 \[
\rho_n \vu_n \rightharpoonup \rho\vu  \,\, \mbox{weakly in} \,\, L^p((0,T; L^{r}(\Omega)),\  1\le p < +\infty,\ 1\leq r<2
 \]
 \[
\rho_n \vu_n\otimes \vu_n \rightharpoonup \rho\vu\otimes\vu \,\, \mbox{weakly in} \,\, L^p((0,T; L^{1}(\Omega)),1\le p < +\infty
 \]
 \[
 \psi_n \rightarrow \psi \,\, \mbox{strongly in} \,\, L^p((0,T; L^{1}(\Omega\times D)), \,\, 1\le p < +\infty,
 \]
 \[
 \eta_n \rightarrow \eta  \,\, \mbox{strongly in} \,\, L^{2}(0,T;L^{2}(\Omega))
 \]
 $0\le \rho \le 1$
and 
 $(\rho,  \vu, \psi, \eta)$ is a  weak solution to the problem \eqref{1.14}-\eqref{1.20} in the sense of Definition \ref{def1}. 
 \label{MT2}
 \end{theorem}
 
 Finally, we want to point out that the solution we are going to construct satisfies the following energy inequality
\[
\frac{d}{dt} \int_{\Omega} \left[\frac{1}{2} \rho |\vu|^2 +\xi \eta^{2} + k(l-(K+1)) {\mathcal   F}(\eta) + k \int_D M {\mathcal F}\left(\frac{\psi}{M}\right) d \vc{q}\right] dx 
\]
\[
+\mu^S \int_{\Omega} |D(\vu) - \frac{1}{d} (\Div_x \vu) \mathbb{I}|^2 dx + \mu^B \int_{\Omega} |\Div_x \vu|^2 dx 
\]
\[
+ 2 \e \xi \int_{\Omega} |\Grad_x \eta|^2 dx + 4 \e k (L - (K+1)) \int_{\Omega} |\nabla_x \sqrt{\eta}|^2 dx
\]
\[
+ 4k \e \int_{\Omega \times D} M |\Grad_x \sqrt{\frac{\psi}{M}}|^2 d \vc{q} dx 
\]
\begin{equation}
+ \frac{k}{\lambda} \sum_{i=1}^K \sum_{j=1}^K \int_{\Omega \times D} M \nabla_{\vc{q}_j} \sqrt{\frac{\psi}{M}} \nabla_{\vc{q}_i}  \sqrt{\frac{\psi}{M}} d \vc{q} dx= \int_{\Omega} \rho f \cdot \vc{u} dx.
\label{energy}
\end{equation}

The rest of the paper is devoted to the proof of the Theorems \ref{MT} and  \ref{MT2}.

\section{Formulation of the approximating problem}\label{S3}

\subsection{The approximating scheme.}
As already mentioned the solutions of the problem  {$\mathbf{(P_F)}$} will be obtained by means of an  an approximating procedure. In this section we will 
set up   the approximating scheme we are going to use. 

Let be $\gamma_{n}$ a sequence of real numbers such that $\gamma_{n}>\frac{3}{2}$, for any $n \in \N$ and $\gamma_{n}\to \infty$ as 
$n\to \infty$, we define  $\{\rho_{n}, \vu_{n}, \psi_{n}, \eta_{n}\}$ as solutions of the following system denoted as ${\mathbf{(P_n)}}$.

\begin{equation} \label{3.1}
\partial_t \rho_n+\Div(\rho_n \vu_n)=0
 \end{equation}
\begin{equation}\label{3.2}
\partial_t (\rho_n  \vu_n) +\Div(\rho_n \vu_n \otimes \vu_n)+ \nabla_x (p_{Fn}+\xi\eta_{n}^{2})= \Div \mathbb{S}[\vc{u_n}, \rho_n] + \rho f + \Div \vc{\tau_{1n}}
\end{equation}
\begin{equation}\label{3.3}
 \partial_t \psi_n+\Div(\psi_n \vu_n) + \sum_{i=1}^K \nabla_{\vc{q_i}} \cdot (\vc{\sigma(u_n)} \vc{q_i}(\psi_n)) 
 \end{equation}
\[
= \eps  \Delta_{x}\left(\frac{\psi_n}{\zeta(\rho_n)}\right)\frac{1}{4\lambda} \sum_{i=1}^K \sum_{j=1}^K A_{i,j} \Grad_{\vc{q_i}} \cdot 
\left(M \Grad_{\vc{q_i}} \left(\frac{\psi_n}{\zeta(\rho_n)  M}\right)\right) 
\]

\begin{equation}
\partial_t \eta_n+\nabla\cdot(\eta_n\vu_n) - \e \Delta_x \left(\frac{\eta_n}{\zeta(\rho_n}\right) =0, \label{3.4}
\end{equation}
where $\vc{\tau_{1n}}=\vc{\tau_{1}}(\psi_{n})$ and 
$$p_{Fn}=(\rho_{n})^{\gamma_{n}},\  \gamma_{n}\to \infty,\ \text{as $n\to \infty$}.$$

The approximating system must be complemented with boundary and initial data as follows.

\subsubsection{Boundary data}
\begin{equation}
 \label{b1}
 \vu_{n}=0, \,\,  \psi_{n}=0,\, \mbox{and} \, \, \eta_{n}=0, \, \mbox{ on} \,\  \partial \Omega.
 \end{equation}
\vspace{0.1in}
\subsubsection{Initial data}
\begin{equation}
\rho_n|_{t=0} = \rho_{n_0}, \quad \rho_n \vu_n|_{t=0} = m_{n_0}, \quad \eta_{n}|_{t=0}=\eta_{n_0},\quad \psi_{n}|_{t=0}=\psi_{n_0}
\label{i1}
\end{equation}
where
\begin{eqnarray}
& &  \rho_{n_0}\geq 0 \quad \mbox{a.e}, \quad  \rho_{n_{0}} \in L^1(\Omega) \cap L^{\gamma_n}(\Omega), \nonumber\\
& & \int (\rho_{n_0})^{\gamma_n} dx \le c \gamma_n \,\, \mbox{for some}\,\, c, \label{i5}\\
& & \quad \quad \quad  m_{n_0} \in L^{\frac{2 \gamma_n}{\gamma_n + 1}} (\Omega), \nonumber \\
&& \rho_{n_0} |\vu_{n_0}|^2  \,\, \mbox{is bounded in}\,\,  L^1(\Omega), \nonumber\\
&& \quad \quad \quad \vu_{n_0} = \frac{m_{n_0}}{\rho_{n_0}}\,\, \mbox{on} \,\, \{\rho_{n_0} >0\}, \nonumber\\
&&  \quad \quad \quad \vu_{n_0} = 0\,\, \mbox{on} \,\, \{\rho_{n_0} = 0\}, \nonumber\\
&&  \quad \quad \quad  \psi_{n_0}  \in L^1(\Omega \times S^{2}) \nonumber\\
&&  \quad \quad \quad \eta_{n_0}  \in L^2(\Omega \times S^{2}).\nonumber
\end{eqnarray}
Furthermore we assume that
\begin{equation}
V_{n}=\int_{\Omega}\!\!\!\!\!\!\!-\rho_{n_0},\quad  0<V_{n}<V<1,\quad V_{n}\to V,
\label{i4}
\end{equation}
and
\begin{equation}\nonumber
\rho_{n_0}\vu_{n}\rightharpoonup m_{0}\quad \text{weakly in $L^{2}(\Omega)$,}
\end{equation}
\begin{equation}
\rho_{n_0}\rightharpoonup \rho_{0}\quad \text{weakly in $L^{1}(\Omega)$.}
\label{i3}
\end{equation}
\subsection{Existence of approximate solutions}

 For any fixed $n\in \N$, the existence of weak solutions for the system \eqref{3.1}-\eqref{3.4} has been proved by Barrett and S\"{u}li (2016) \cite{BS2016} (we refer the reader also to a series of earlier works on related models \cite{BS2011, BS2012A,BS2012B}).
 
 We can summarize their existence result as follows.
\begin{theorem} \label{thm:2.2}
The triple $(\rho_n, \vc{u}_{n}, \widehat{\psi}_n),$ is a global weak solution to problem ${\mathbf{(P_n)}}$ in the sense that 
the following  relations hold true

\begin{equation}
\int_0^T \Bigl< \frac{\partial \rho_n}{\partial t}, \chi \Bigr>_{W^{1,6}(\Omega)} dt - \int_0^T \int_{\Omega} \rho_n \vu_n \cdot \Grad_x \chi  dx dt =0, \label{ceq} 
\end{equation}
for any  $\chi \in L^2(0,T; W^{1,6}(\Omega))$ with $\rho_n(\cdot, 0) = \rho_{n,0}(\cdot),$

\[
\int_0^T \Bigl< \frac{\partial (\rho_n \vu_n)}{\partial t}, \vc{w} \Bigr>_{W^{1,r}_0(\Omega)} dt + \int_0^T \int_{\Omega} [\mathbb{S}(\vu_n) - \rho_n \vu_n \otimes \vu_n - c_p \rho_n^{\gamma_{n}} \mathbb{I}]: \Grad_x \vc{w} dx dt \nonumber
\]
 \begin{equation}
 = \int_0^T \int_{\Omega} [\rho_n \vc{f} \cdot \vc{w} - (\vc{\tau_{1}}(M \widehat{\psi_{n}}) - \xi \eta_n^2 \mathbb{I}): \Grad_x \vc{w} ]dx dt   \nonumber
 \end{equation}
 \begin{equation} 
\text{for all}\   \vc{w} \in L^{\frac{\gamma + \vartheta}{\vartheta}}(0,T; W^{1,r}_0(\Omega)) \label{meq} 
\end{equation}
with $(\rho_n \vc{u_n})(\cdot, 0) = (\rho_{n,0} \vc{u_{n,0}})(\cdot)$ and  $\vartheta(\gamma)$ is defined as 
$$\vartheta(\gamma):= \frac{\gamma}{v(\gamma)} = 
\begin{cases}
& \frac{2\gamma - 3}{3} \,\, \mbox{for} \,\, \frac{3}{2} < \gamma \le 4,\\
& \frac{5}{12} \gamma    \,\, \mbox{for} \,\, 4 \le \gamma
\end{cases}
$$
and $r = \max \left\{4, \frac{6 \gamma}{2\gamma - 3}\right\}.$
\[
\int_0^T \Bigl< M \frac{\partial \widehat{\psi}_n}{\partial t}, \varphi \Bigr>_{H^s(\Omega \times D)} dt +
\frac{1}{4\lambda} \sum_{i=1}^K \sum_{j=1}^K A_{ij} \int_0^T \int_{\Omega \times D} M \nabla_{\vc{q_j}} \widehat{\psi}_n \cdot \nabla_{\vc{q_j}} \varphi d \vc{q} dx dt
\]
\[
+ \int_0^T \int_{\Omega \times D} M [ \e \Grad_x \widehat{\psi}_n - \vu  \widehat{\psi}_n] \cdot \Grad_x \varphi  d\vc{q} dx dt
\]
\[
+ \int_0^T \int_{\Omega \times D} M \sum_{i=1}^K [\sigma(\vu_n) \vc{q_i}] \widehat{\psi} \cdot \Grad_{\vc{q_i}} \varphi d \vc{q} dx dt =0,
\]
\begin{equation}
\text{for all}\  \varphi \in L^2(0,T; H^s(\Omega \times D)). \label{pdfeq} 
\end{equation}
In addition, the weak solution $(\rho_{n}, \vc{u}_{n}, \widehat{\psi_{n}})$ satisfies the inequality 
\[ \frac{1}{2} \int_{\Omega} \rho_{n}(t') |\vu_{n}(t')|^2 dx + \int_{\Omega} P(\rho_{n}(t')) dx + k \int_{\Omega \times D} M \mathcal{F}(\widehat{\psi}_{n}(t')) d \vc{q} dx 
\]
\[
+ \mu^S c_0 \int_0^{t'} \|\vu_{n}\|^2_{H^1(\Omega)}dt
\]
\[
+ k \int_0^{t'} \int_{\Omega \times D} M \left[\frac{a_0}{2\lambda}|\Grad_{\vc{q}}|\sqrt{\widehat{\psi_{n}}}|^2 + 2\e |\Grad_x \sqrt{\widehat{\psi_{n}}}|^2 \right]d\vc{q} dx dt
\]
\[
+ \xi \|\eta_{n}(t')\|^2_{L^2(\Omega)} + 2 \xi \e  \int_0^{t'} \|\Grad_x \eta_{n}\|_{L^2(\Omega)}dt 
\]
\[
\le e^{t'} \left[\frac{1}{2} \rho_0|\vc{u_0}|^2 dx + \int_{\Omega} P(\rho_0)dx + k \int_{\Omega \times D} M \mathcal{F}(\widehat{\psi}_0) d \vc{q} dx \right.
\]
\begin{equation}
\left. +\xi \int_{\Omega} \left( \int_D M \widehat{\psi}_0 d\vc{q} \right)^2 dx + \frac{1}{2} \int_0^{t'} \|f\|^2_{L^{\infty}(\Omega)} dt \int_{\Omega} \rho_0 dx\right].
\label{energyapprox}
\end{equation}

\end{theorem}
\begin{proof}
For the sake of completeness we present now an outline of the proof. For the details we refer the reader to Barrett and S\"{u}li \cite{BS2016}.
The proof relies on two key observations:
\begin{enumerate}[\quad $\bullet$]
\item If $\displaystyle{\frac{\psi}{M}}$ is bounded above then, for $L \in \mathbb{R}_+$ sufficiently large, the third term in \eqref{1.3} is equal to
$$\sum_{i=1}^K \Grad_{\vc{q_i}} \cdot \left( \vc{\sigma}(\vu) \vc{q_i} M \beta^L\left(\frac{\psi}{M}\right)\right),$$
where $\zeta = 1$ and $\beta^L \in C(\mathbb{R})$ denotes a cut-off function such as $\beta^L(s):= \min\{s, L\}.$
It follows that for sufficiently large $L$ any solution of \eqref{1.3} such as $\displaystyle{\frac{\psi}{M}}$ is bounded above by $L$ also satisfies
\begin{equation}\label{3.15}
 \partial_t \psi +\Div(\psi  \vu) + \sum_{i=1}^K \Grad_{\vc{q_i}} \cdot \left( \vc{\sigma}(\vu) \vc{q_i} M \beta^L\left(\frac{\psi}{M}\right)\right)
 \end{equation}
\[
= \eps  \Delta_{x}\left(\frac{\psi}{\zeta(\rho)}\right)\frac{1}{4\lambda} \sum_{i=1}^K \sum_{j=1}^K A_{i,j} \Grad_{\vc{q_i}} \cdot 
\left(M \Grad_{\vc{q_i}} \left(\frac{\psi}{ M}\right)\right) 
\]
in $\Omega \times D \times (0,T]$ supplemented  with the following boundary conditions: 
\begin{equation}
\left[\frac{1}{4\lambda} \sum_{j=1}^K A_{i,j} M \Grad_{\vc{q_j}} \left(\frac{\psi}{M}\right) - \vc{\sigma}(\vu) \vc{q_i} M \beta^L\left(\frac{\psi}{M}\right)\right] \cdot \frac{\vc{q_i}}{|\vc{q_i}|} = 0 \label{3.16}
\end{equation}
on $\Omega \times \partial \bar{D}_i \times (0,T]$  for  $ i= 1, \dots, K.$
\begin{equation}\label{3.16a}
\e \Grad_x \psi \cdot \vc{n} = 0 \,\,\,\, \mbox{on}\,\, \partial \Omega \times D \times (0,T]
\end{equation}
and the initial conditions: 
\begin{equation}
\psi(\cdot, \cdot, 0) = M(\cdot)\beta^L\left(\frac{\psi_0(\cdot,\cdot)}{M(\cdot)}\right) \ge 0,\,\, \mbox{on}\, \Omega \times D.
\end{equation}
The model with cut-off parameter $L > 1$ is further regularized, by introducing into the continuity equation a dissipation term of the form $\alpha \Delta \rho,$ with $\alpha > 0$  and supplementing the resulting parabolic equation with a homogeneous Neumann boundary condition on $\partial \Omega \times (0, T ].$ Moreover, the equation of state (1.3) is replaced by a regularized equation of state, 
$$p_{\kappa}(\rho) = p(\rho)+ \kappa(\rho^4 + \rho^{\Gamma}), \,\,  \mbox{with}\,\,  \kappa \in \mathbb{R}_+ \,\,\, \Gamma = \max\{\gamma, 8\}.$$ 
\item  The second key element of the proof is that instead of taking the limits $\kappa \to 0_+, \alpha \to 0_+, L \to +\infty$ to deduce the existence of solutions to $\vc{(P_n)}$   the authors  (semi)discretize the problem with respect to $t,$ with step size $\Delta t.$ The existence of solutions to this problem is established by employing  Schauder�s fixed point theorem. 
Next one derives bounds on the sequence of solutions to the time discretized problem problem uniform in the time step $\Delta t$ and the cut-off  parameter $L,$ and thus permit the extraction of weakly convergent subsequences, as $L \to \infty$ and $\Delta t \to 0_+$, with $\Delta t = o(L-1),$ when $L \to \infty$. The weakly convergent subsequences are then be shown to converge strongly in suitable norms. This allows to the  passing  to the limit as $L \to + \infty$, with $\Delta t = o(L-1).$  
The result follows by passing  to the limit as $\alpha, \kappa \to 0_+.$ 
\end{enumerate}
\end{proof}

\subsection{A priori estimates and compactness for the approximating sequences}
\label{sec-estimates}
We start this section by collecting all the a priori estimates that can be deduced  by  the energy estimate \eqref{energyapprox}. In particular for any fixed 
 $\gamma_{n}>3/2$  we have

$$\rho_{n} \in L^{\infty}(0,T; L^{\gamma_{n}}(\Omega)), \quad \nabla \vu_n \in L^{2}(0;T; L^{2}(\Omega)),$$
$$\sqrt{\rho_n}\vu_n\in L^{\infty}(0,T; L^{2}(\Omega)), \quad \rho_n \vu_n \in C_{w}([0,T];L^{\frac{2\gamma_{n}}{\gamma_{n}+1}}(\Omega)),$$
$$\eta_n \in L^{\infty}(0,T; L^{2}(\Omega))\cap L^{2}(0,T; \dot{H}^{1}(\Omega))$$
$$\mathcal{F}(\widehat{\psi}_{n}) \in L^{\infty}(0,T; L^{1}(\Omega\times D))$$
$$M^{1/2}\nabla _{x}\sqrt{\widehat{\psi}_{n}} \in L^{2}(0,T;L^{2}(\Omega\times D)), \quad M^{1/2}\nabla_{q}\sqrt{\widehat{\psi}_{n}} \in L^{2}(0,T;L^{2}(\Omega\times D)),$$
By using \eqref{3.3} and \eqref{energyapprox} we get also that
$$M\frac{\partial \widehat{\psi}_{n}}{\partial t}\in  L^{2}(0,T;H^{s}(\Omega\times D)')$$
Moreover, by applying well established techniques (see for example \cite{BS2016} or \cite{DonatelliTrivisa2017}) we  are  able to show to following uniform bound for the density
\begin{equation}\nonumber
\rho_{n}\in  L^{\infty}(0,T; L^{1}\cap L^{\Gamma}(\Omega)),\qquad \text{for any} \ \Gamma\geq 8.
\end{equation}
With the preceeding bounds we can get some further estimates on the elastic stress tensor
$$ \vc{\tau(\psi)}(x,t) := \vc{\tau_1(\psi)}(x,t) - \xi \left(\int_D  \psi d\vc{q}\right)^2 \mathbb{I},$$
where
$$\vc{\tau_1}(\psi):= k \left[\left(\sum_{i=1}^K \vc{C_i}(\psi)\right) - (K+1) \int_D \psi(\vc{q}) d\vc{q}\, \mathbb{I}\right], $$
In fact by following the same lines of arguments as in \cite{BS2016} it is possible to prove that
$$\|\vc{C_i}(M\widehat{\psi}_{n})\|_{L^{2}(0,T;L^{\frac{4}{3}}(\Omega))}+\|\vc{C_i}(M\widehat{\psi}_{n})\|_{L^{\frac{4(d+2)}{3d+4}}(0,T;\Omega)}\leq C$$
from which it is straightforward to deduce that
\begin{equation}
\label{tauone}
\|\vc{\tau_1}(M\widehat{\psi}_{n})\|_{L^{2}(0,T;L^{\frac{4}{3}}(\Omega))}+\|\vc{\tau_1}(M\widehat{\psi}_{n})\|_{L^{\frac{4(d+2)}{3d+4}}(0,T;\Omega)}\leq C
\end{equation}
From the previous a priori estimates, extracting a subsequence, we can deduce  the following convergence results
\begin{equation} 
 \begin{split}
 & \rho_{n} \rightharpoonup \rho \hspace{0.2cm} \text{ weakly in} \hspace{0.2cm} L^{\infty}(0,T;L^{p}(\Omega)\!),\ \rho \in L^{\infty}(0,T; L^{1}\cap L^{p}(\Omega)\!),\ 1\leq p<\!\!+\infty,\\
 & \sqrt{\rho_{n}} \vu_{n} \rightharpoonup \sqrt{\rho}\vu \hspace{0.2cm} \text{ weakly in} \hspace{0.2cm} L^{2}(0,T; L^{2}(\Omega)),\\
 & \vu_{n} \rightharpoonup \vu \hspace{0.2cm} \text{weakly in} \hspace{0.2cm} L^{2}(0,T; H^{1}(\Omega)),\\
 & \rho_{n} \vu_{n} \rightharpoonup \rho \vu \hspace{0.2cm} \text{weakly in} \hspace{0.2cm}L^{2}(0,T; L^{\frac{6\gamma}{\gamma+6}}(\Omega)),\\
 & \rho_{n}\vu_{n}\otimes\vu_{n} \rightharpoonup  \rho\vu\otimes\vu\ \hspace{0.2cm} \text{in weakly}\hspace{0.2cm}L^{2}(0,T; L^{\frac{6\gamma}{4\gamma+3}}(\Omega)), \\
  & \eta_{n} \rightharpoonup \eta \hspace{0.2cm} \text{weakly in} \hspace{0.2cm} L^{2}(0,T; H^{1}(\Omega)\!),\quad \!\!\!\eta \in L^{\infty}(0,T; L^{2}(\Omega)\!) \cap L^{2}(0,T; H^{1}(\Omega)\!),\\
& M^{1/2}\nabla _{x}\sqrt{\widehat{\psi}_{n}} \rightharpoonup M^{1/2}\nabla _{x}\sqrt{\widehat{\psi}}\hspace{0.2cm} \text{weakly in} \hspace{0.2cm} L^{2}(0,T;L^{2}(\Omega\times D)),\\
& M^{1/2}\nabla _{q}\sqrt{\widehat{\psi}_{n}} \rightharpoonup M^{1/2}\nabla _{q}\sqrt{\widehat{\psi}}\hspace{0.2cm} \text{weakly in} \hspace{0.2cm} L^{2}(0,T;L^{2}(\Omega\times D)),\\
& M\frac{\partial \widehat{\psi}_{n}}{\partial t} \rightharpoonup M\frac{\partial \widehat{\psi}}{\partial t}\hspace{0.2cm} \text{weakly in} \hspace{0.2cm}   L^{2}(0,T;H^{s}(\Omega\times D)'),\\
&\psi_{n} \rightarrow \psi\hspace{0.2cm} \text{ strongly in} \hspace{0.2cm} L^{p}(0,T; L^{1}(\Omega\times D)),\ p\geq 1,\\
&  \vc{\tau}(\widehat{\psi}_{n})\rightarrow \vc{\tau}(\widehat{\psi})\hspace{0.2cm} \text{strongly in} \hspace{0.2cm} L^{r}(0,T;\Omega)
\end{split}
  \label{eq:2.16}
\end{equation}

\section{Proof of the Main Theorem \ref{MT}}\label{S4}
In this section we are going to prove the existence of a global weak solution for the problem ${\mathbf{(P_F)}}$,  the Main Theorem \ref{MT}. We start by showing a stability result for the approximating sequences,  Theorem \ref{MT2}.
\subsection{Proof of the  Theorem  \ref{MT2}}
One of the main issues in the proof is to get a uniform $L^{1}$ bound  in $n\in \N$  for $\rho_{n}^{\gamma_{n}}$. Indeed from \eqref{energyapprox} we only have $\int \rho_{n}^{\gamma_{n}}\leq C(\gamma_{n}-1)$.

 For simplicity we divide the proof in different steps. \\[0.5cm]
{\bf Step 1: Convergence of  $\mathbf{(\rho_{n}-1)_{+}}$ to $0$.}\\
The energy inequality \eqref{energyapprox} with the initial condition \eqref{i5} gives
\begin{equation}
\label{e1}
\int_{\Omega}(\rho^{\gamma_{n}})_{n} dx\leq (\gamma_{n}-1)E_{n_{0}}+\int_{\Omega} (\rho_{n_0})^{\gamma_n} dx \leq (\gamma_{n}-1)E_{n_{0}}+c \gamma_n
\leq c\gamma_{n}.
\end{equation}
Since $\gamma_{n}\to \infty$ there exists  $n\in \N$ such that $\gamma_{n}>p$,  $1<p<+\infty$, then by H\"older inequality we get
\begin{equation}\nonumber
\|\rho_{n}\|_{L^{\infty}_{t}L^{p}_{x}}\leq \|\rho_{n}\|_{L^{\infty}_{t}L^{1}_{x}}^{\theta_{n}} \|\rho_{n}\|_{L^{\infty}_{t}L^{1}_{x}}^{1-\gamma_{n}}\leq
V_{n}^{\theta_{n}}(c\gamma_{n})^{\frac{1-\theta_{n}}{\gamma_{n}}},
\end{equation}
where $V_{n}$ is defined in \eqref{i4} and $\displaystyle{\frac{1}{p}=\theta_{n}+\frac{1-\theta_{n}}{\gamma_{n}}}$.  We have that $\displaystyle{\theta_{n}\to \frac{1}{p}}$, as $n\to \infty$ and we end up with
\begin{equation}\nonumber
\|\rho_{n}\|_{L^{\infty}_{t}L^{p}_{x}}\leq \liminf_{n\to\infty}\|\rho_{n}\|_{L^{\infty}_{t}L^{p}_{x}}\leq M^{1/p}.
\end{equation}
We define the function $\phi_{n}$ as follows
$$\phi_{n}=(\rho_{n}-1)_{+},$$
by using again the energy inequality \eqref{energyapprox} we can compute
\begin{equation}
\label{e4}
\int_{\Omega}(1+\phi_{n})^{\gamma_{n}}\mathbf{1}_{\{\phi_{n}>0\}}dx \leq \int_{\Omega}\rho^{\gamma_{n}} dx\leq c\gamma_{n}.
\end{equation}
We apply the inequality
\begin{equation}\nonumber
(1+x)^{k}\geq 1+c_{p}k^{p}x^{p}, \quad p>1, k \ \text{large}, \ x>0
\end{equation}
 with $k=\gamma_{n}$, $x=\phi_{n}$ to the right hand side of \eqref{e4}, so we obtain
\begin{equation}\nonumber
c_{p}\gamma_{n}^{p}\int_{\Omega}\phi_{n}^{p}dx\leq |\Omega|+c_{p}\gamma_{n}^{p}\int_{\Omega}\phi_{n}^{p}dx\leq\int_{\Omega}(1+\phi_{n})^{\gamma}\mathbf{1}_{\{\phi_{n}>0\}}dx\leq c\gamma_{n}.
\end{equation}
Therefore we have
\begin{equation}\nonumber
\int_{\Omega}\phi_{n}^{p}dx\leq \frac{c}{c_{p}\gamma_{n}^{p-1}},
\end{equation}
and, as $n\to \infty$ we get
\begin{equation}\nonumber
(\rho_{n}-1)_{+}\rightarrow 0 \qquad \text{in $L^{\infty}(0,T;L^{p}(\Omega))$,  $1\leq p<+\infty$.}
\end{equation}\\[0.5cm]

{\bf Step 2: $\mathbf{L^{1}}$ uniform bound of $\mathbf{(\rho_{n})^{\gamma_{n}}}$.}\\
In order to prove a uniform bound for the pressure $p_{F}$, we start  by assuming that we know 
\begin{equation}
(\rho_{n})^{\gamma_{n}+1} \quad \text{is uniformly bounded in $L^{1}(0,T;L^{1}(\Omega))$},
\label{e8}
\end{equation}
hence we have
\begin{equation}
\begin{split}
\int_{0}^{T}\!\!\int_{\Omega}(\rho_{n})^{\gamma_{n}}dxdt&=\int_{0}^{T}\!\!\left(\int_{\Omega\cap\{\rho_{n}>1\}}(\rho_{n})^{\gamma_{n}}dx+\int_{\Omega\cap\{\rho_{n}\leq1\}}(\rho_{n})^{\gamma_{n}}dx\right)dt\\
&\leq \int_{0}^{T}\!\!\left(\int_{\Omega}\left((\rho_{n})^{\gamma_{n}+1} +\rho_{n}\right)dx \right)dt.
\end{split}
\label{e9}
\end{equation}
Since $\rho_{n}\in L^{\infty}(0,T;L^{1}(\Omega))$ and  \eqref{e8} holds, from \eqref{e9} it follows the uniform $L^{1}$ bound for $(\rho_{n})^{\gamma_{n}}$.

The only thing we need to  prove is  \eqref{e8}. We recall  that for $\rho_{n}$ we don't have $L^{\infty}$ bounds, but on the other hand, 
because of \eqref{e1} there exists a constant $\tilde{c}$ such that for any $n\in \N$ the following estimate holds
\begin{equation}
\|\rho_{n}\|_{L^{\infty}_{t}L^{\gamma_{n}}_{x}}\leq \tilde{c},
\label{e10}
\end{equation}
where $\displaystyle{\tilde{c}=\sup_{\gamma>0}(c\gamma)^{1/\gamma}}$.

Let us   define, now, the operator $\mathcal{B}$ as the inverse of the divergence operator. We denote the solution $v$ of
\[
\Div v=g \hspace{0.2cm} \text{in} \hspace{0.2cm} \Omega, \quad v=0 \hspace{0.2cm} \text{on} \hspace{0.2cm} \partial \Omega.
\]
by $v=\mathcal{B}g$. The operator $\mathcal{B}=( \mathcal{B}_{1}, \mathcal{B}_{2}, \mathcal{B}_{3})$  enjoys the following properties
\[
\mathcal{B}: \Big\{g\in L^{p}; \int_{\Omega}g dx=0 \Big\} \rightarrow W^{1,p}_{0}(\Omega),
\]
\[
\|\mathcal{B}(g)\|_{W^{1,p}(\Omega)} \leq C \|g\|_{L^{p}(\Omega)}.
\]
If $g$ can be written as $g=\Div h$ for a certain $h \in L^{r}$ with $h\cdot \hat{n}=0$ on $\partial\Omega$, then
\[
\|\mathcal{B}(g)\|_{L^{r}(\Omega)} \leq C \|h\|_{L^{r}(\Omega)}.
\]
With same lines of arguments as in \cite{DonatelliTrivisa2017} it is possible to prove that our approximating sequences satisfy also the equation \eqref{3.1} in the sense of renormalized solutions, namely 
\begin{equation} 
\label{eq:4.1}
\partial_{t}b(\rho_{n})_{\epsilon}+\Div (b(\rho_{n})_{\epsilon}u)+\Big(\big[ b'(\rho_{n})\rho_{n}-b(\rho_{n})\big] \Div \vu_{n} \Big)_{\epsilon}=r_{\epsilon},
\end{equation}
where as proved in Lions \cite{Lions1998}, $r_{\epsilon} \rightarrow 0$ in $L^{2}((0,T)\times\R^{3})$. Let us  take a test function of the form
$$\phi_{i}=\chi(t)\mathcal{B}_{i}\Big[b(\rho_{n})_{\epsilon}- \oint_{\Omega}b(\rho_{n})_{\epsilon}dy\Big],$$ 
where
\[ \oint_{\Omega}b(\rho_{n})_{\epsilon}dy=\frac{1}{|\Omega|}\int_{\Omega}b(\rho_{n})_{\epsilon}dy, \quad \chi \in \mathcal{D}(0,T)
\]
and test it against (\ref{3.2}). Then, with the aid of (\ref{eq:4.1}) and by setting for simplicity all the constants equal to 1 we can compute
\begin{equation*}
 \begin{split}
 \int^{T}_{0}\!\!\int_{\Omega} \chi \rho_{n}^{\gamma_{n}}b(\rho_{n})_{\epsilon}dxdt 
 &= \int^{T}_{0}\!\!\int_{\Omega} \chi \rho_{n}^{\gamma_{n}} \Big[\oint_{\Omega}b(\rho_{n})_{\epsilon}dy \Big]dxdt \\
 &-\int^{T}_{0}\!\!\int_{\Omega} \chi_{t}\rho_{n} \vu_{n} \cdot \mathcal{B}\Big[ b(\rho_{n})_{\epsilon}-\oint_{\Omega}b(\rho_{n})_{\epsilon}dy\Big] dxdt\\
 &+\int^{T}_{0}\!\!\int_{\Omega} \chi \rho_{n} \vu_{n}  \cdot \mathcal{B}\Big[ \big( (b^{'}(\rho_{n})\rho_{n}-b(\rho_{n}))\Div \vu_{n} \big)_{\epsilon}\\
 & -\oint_{\Omega} \big( (b^{'}(\rho_{n})\rho_{n}-b(\rho_{n}))\Div \vu_{n} \big)_{\epsilon}dy\Big]dxdt\\
 & -\int^{T}_{0}\!\!\int_{\Omega} \chi \rho_{n} \vu_{n}  \cdot \mathcal{B}\Big[ r_{\epsilon} -\oint_{\Omega}r_{\epsilon}dy\Big]dxdt\\
 &+ \int^{T}_{0}\!\!\int_{\Omega}\chi \rho_{n}\vu_{n}  \cdot \mathcal{B}\Big[\nabla \cdot \big(b(\rho_{n})_{\epsilon} \vu_{n}  \big)\Big]dxdt\\
 &- \int^{T}_{0}\!\!\int_{\Omega}\chi \rho_{n} \vu_{ni}\vu_{nj} \partial_{i}\mathcal{B}_{j}\Big[ b(\rho_{n})_{\epsilon} -\oint_{\Omega}b(\rho_{n})_{\epsilon}dy\Big]dxdt \\
 & + \int^{T}_{0}\!\!\int_{\Omega}\chi \partial_{i}\vu_{nj}\partial_{i} \mathcal{B}_{j} \Big[ b(\rho_{n})_{\epsilon} -\oint_{\Omega}b(\rho_{n})_{\epsilon}dy\Big]dxdt\\
 & + \int^{T}_{0}\!\!\int_{\Omega}\chi \Div \vu_{n}  \Big[ b(\rho_{n})_{\epsilon} -\oint_{\Omega}b(\rho_{n})_{\epsilon}dy\Big]dxdt \\
 &- \int^{T}_{0}\!\!\int_{\Omega}\chi \eta^{2}_{n}\Big[ b(\rho_{n})_{\epsilon} -\oint_{\Omega}b(\rho_{n})_{\epsilon}dy\Big]dxdt\\
 & + \int^{T}_{0}\!\!\int_{\Omega}\chi \vc{\tau_{1}}_{ij}(\psi_{n})\partial_{i} \mathcal{B}_{j} \Big[ b(\rho_{n})_{\epsilon} -\oint_{\Omega}b(\rho_{n})_{\epsilon}dy\Big]dxdt \\
 &=I_{1} +\cdots +I_{11}.
 \end{split}
\end{equation*}
By using the properties of the operator $\mathcal{B}$, the a priori bounds of the Section \ref{sec-estimates} and \eqref{e10}  we estimate  each one of the terms $I_{1}, \cdots, I_{11}$. For details, 
see Feireisl \cite{DonatelliTrivisa2017}.\\

For $I_{1}$ we have
$$I_{1} \lesssim C(T).$$
Concerning $I_{2}$ we get
\begin{equation*}
 \begin{split}
I_{2} &\lesssim  \|\rho_{n} \vu_{n} \|_{L^{\infty}(0,T; L^{\frac{2\gamma_{n}}{\gamma_{n}+1}}(\Omega))} \|b(\rho_{n})_{\epsilon}\|_{L^{\infty}(0,T; L^{\frac{6\gamma_{n}}{5\gamma_{n}-3}}(\Omega))}\\
&\leq C(T) \|b(\rho_{n})_{\epsilon}\|_{L^{\infty}(0,T; L^{\frac{6\gamma_{n}}{5\gamma_{n}-3}}(\Omega))}.
\end{split}
\end{equation*}
For $I_{3}$, $I_{4}$  and $I_{5}$ we have,
\begin{equation*}
 \begin{split}
 I_{3} +I_{4}&\lesssim  \|\rho_{n}\|_{L^{\infty}(0,T; L^{\gamma}(\Omega))} \|\nabla \vu_{n} \|^{2}_{L^{2}(\Omega\times(0,T))} \|b(\rho_{n})_{\epsilon}\|_{L^{\infty}(0,T;L^{\frac{3\gamma_{n}}{2\gamma_{n}-3}}(\Omega) )} \\
 &\leq C(T) \|b(\rho_{n})_{\epsilon}\|_{L^{\infty}(0,T;L^{\frac{3\gamma_{n}}{2\gamma_{n}-3}}(\Omega) )}.
 \end{split}
\end{equation*}
$$I_{5} \lesssim \|\rho_{n} \vu_{n} \|_{L^{\infty}(0,T;L^{\frac{2\gamma_{n}}{\gamma_{n}+1}}(\Omega))} \|r_{\epsilon}\|_{L^{2}(\Omega\times(0,T))} \leq C(T)\|r_{\epsilon}\|_{L^{2}(\Omega\times(0,T))}.$$
We estimate now $I_{6}+I_{7}$,
\begin{equation*}
 \begin{split}
 I_{6} + I_{7} &\lesssim  \|\rho_{n}\|_{L^{\infty}(0,T; L^{\gamma_{n}}(\Omega))} \|\nabla \vu_{n} \|^{2}_{L^{2}(\Omega\times(0,T))} \|b(\rho_{n})_{\epsilon}\|_{L^{\infty}(0,T;L^{\frac{3\gamma_{n}}{2\gamma_{n}-3}}(\Omega) )} \\
 &\leq C(T) \|b(\rho_{n})_{\epsilon}\|_{L^{\infty}(0,T;L^{\frac{3\gamma_{n}}{2\gamma_{n}-3}}(\Omega) )}.
 \end{split}
\end{equation*}
For  the terms $I_{8}$,  $I_{9}$ we have
$$I_{8} + I_{9} \lesssim \|\nabla \vu_{n} \|_{L^{2}(\Omega\times(0,T))} \|b(\rho_{n})_{\epsilon}\|_{L^{2}(\Omega\times(0,T))} \leq C(T)\|b(\rho_{n})_{\epsilon}\|_{L^{2}(\Omega\times(0,T))}.$$
and finally the last two terms can be estimates as follows
\begin{equation*}
 \begin{split}
  I_{10} + I_{11} & \lesssim \|\eta_{n}\|^{2}_{L^{2}(0,T; L^{6}(\Omega))} \|b(\rho_{n})_{\epsilon}\|_{L^{\infty}(0,T; L^{\frac{3}{2}} (\Omega))}\\
  &+\|\vc{\tau_{1}}\psi_{n}\|_{L^{2}(0,T; L^{4/3}(\Omega))}\|b(\rho_{n})_{\epsilon}\|_{L^{2}(0,T; L^{4} (\Omega))}\\
  & \leq C(T) \|b(\rho_{n})_{\epsilon}\|_{L^{\infty}(0,T; L^{4} (\Omega))}
 \end{split}
\end{equation*}
In sum,
\begin{equation} \nonumber
 \begin{split}
 & \int^{T}_{0}\!\!\int_{\Omega} \chi \rho_{n}^{\gamma_{n}}(b(\rho_{n}))_{\epsilon}dxdt \\
 & \leq C(T) + \|b(\rho_{n})_{\epsilon}\|_{L^{\infty}(0,T; L^{\frac{6\gamma_{n}}{5\gamma_{n}-3}}(\Omega))}+ \|b(\rho_{n})_{\epsilon}\|_{L^{\infty}(0,T; L^{\frac{3\gamma_{n}}{2\gamma_{n}-3}}(\Omega))} \\
 & + \|b(\rho_{n})_{\epsilon}\|_{L^{\infty}(0,T; L^{4}(\Omega))} + \|b(\rho_{n})_{\epsilon}\|_{L^{2}(\Omega\times(0,T))}+\|r_{\epsilon}\|_{L^{2}(\Omega\times(0,T))}.
 \end{split}
\end{equation}
By taking the limit $\epsilon \rightarrow 0$,
\begin{equation} \nonumber
 \begin{split}
 & \int^{T}_{0}\!\!\int_{\Omega} \chi \rho_{n}^{\gamma_{n}}b(\rho_{n})dxdt \\
 & \leq C(T) + \|b(\rho_{n})\|_{L^{\infty}(0,T; L^{\frac{6\gamma_{n}}{5\gamma_{n}-3}}(\Omega))}+ \|b(\rho_{n})\|_{L^{\infty}(0,T; L^{\frac{3\gamma_{n}}{2\gamma_{n}-3}}(\Omega))} \\
 & + \|b(\rho_{n})\|_{L^{\infty}(0,T; L^{4}(\Omega))} + \|b(\rho_{n})\|_{L^{2}(\Omega\times(0,T))}.
 \end{split}
\end{equation}
We approximate $z\mapsto z$ by a sequence of $\{b_{n}\}$ in (\ref{eq:4.1}), and approximate $\chi$ to the identity function of $(0,T)$. Then,
\begin{equation} \label{eq:4.5}
 \begin{split}
 \int^{T}_{0}\!\!\int_{\Omega} \rho_{n}^{\gamma_{n}+1}dxdt &\leq C(T) + \|\rho_{n}\|_{L^{\infty}(0,T; L^{\frac{6\gamma_{n}}{5\gamma_{n}-3}}(\Omega))}+ \|\rho_{n}\|_{L^{\infty}(0,T; L^{\frac{3\gamma_{n}}{2\gamma_{n}-3}}(\Omega))} \\
 & + \|\rho_{n}\|_{L^{\infty}(0,T; L^{4}(\Omega))} + \|\rho_{n}\|_{L^{2}(\Omega\times(0,T))}.
 \end{split}
\end{equation}
Since $\gamma_{n}\to \infty$ we can always assume that $\gamma_{n}\geq N=3$, hence  by taking into account that $\rho_{n}\in L^{\infty}(0,T;L^{1}(\Omega))$ and \eqref{e10} we have that the right hand side of \eqref{eq:4.5} is uniformly bounded and we can conclude that 
$$\int^{T}_{0}\!\!\int_{\Omega} \rho_{n}^{\gamma_{n}+1}dxdt \leq C(T)$$ 
which completes the proof of \eqref{e8}.\\[0.5cm]

{\bf Step 3: Convergence of  the approximating scheme.}\\
The  compactness properties of  the approximating sequence $\{\rho_{n}, \vu_{n}, \eta_{n}, \psi_{n}\}$ stated in Section \ref{sec-estimates} and the bounds of the Step 1 and Step 2 entail\begin{equation*}\nonumber
 \begin{split}
 & \rho_{n}\vu_{n} \rightarrow \rho \vu \hspace{0.2cm} \text{in} \hspace{0.2cm} L^{p}(0,T; L^{r}(\Omega)) \hspace{0.2cm} \text{for all} \hspace{0.2cm} 1\leq p<\infty, \quad 1\leq r<2, \\
 & \vu_{n} \rightarrow \vu \hspace{0.2cm} \text{in} \hspace{0.2cm} L^{p}(\Omega\times (0,T))\cap\{\rho_{n}>0\} \hspace{0.2cm} \text{for all} \hspace{0.2cm} 1\leq p<2, \\
 & \vu_{n} \rightarrow \vu \hspace{0.2cm} \text{in} \hspace{0.2cm} L^{2}(\Omega\times (0,T))\cap\{\rho_{n} \ge \delta\} \hspace{0.2cm} \text{for all} \hspace{0.2cm} \delta>0, \\
 & \rho_{n}\vu_{ni}\vu_{nj} \rightarrow \rho \vu_{i}\vu_{j} \hspace{0.2cm} \text{in} \hspace{0.2cm} L^{p}(0,T; L^{1}(\Omega)) \hspace{0.2cm} \text{for all} \hspace{0.2cm} 1\leq p<\infty,\\
&  (\rho_{n})^{\gamma_{n}}\rightharpoonup \pi, \hspace{0.2cm} \text{where $\pi\in \mathcal{M}((0,T)\times\Omega).$}
 \end{split}
\end{equation*}
With the above convergence result  and those one obtained in the Section \ref{sec-estimates} we can pass into the weak limit in the system \eqref{3.1}-\eqref{3.4}, and we get that $\rho, \vu, \eta, \psi$ is a weak solution of the problem $\mathbf{(P_{F})}$ provided we prove the conditions \eqref{1.18}-\eqref{1.20}. This is equivalent to the proof of 
\begin{equation}
\label{e12}
\rho\pi=\pi.
\end{equation}
Setting $s_{n}=\rho_{n}\log\rho_{n}$ and $\bar{s}=\overline{\rho\log\rho}$ 
and using (\ref{3.1}) we get
\begin{equation} \nonumber
(\rho_n \log \rho_n)_{t} +\nabla \cdot(\rho_n \log \rho_n \vu_n)+ (\nabla \cdot \vu_n)\rho_n=0.
\end{equation}
 Then  we apply $(-\Delta)^{-1}\nabla \cdot $ to (\ref{3.2}), 
\[
\frac{d}{dt}\Big[(-\Delta)^{-1}\nabla \cdot(\rho_n \vu_n)\Big]+(-\Delta)^{-1}\partial_{i}\partial_{j}(\rho_n {\vu_n}_{i}{\vu_n}_{j})+2\nabla \cdot \vu_n -\rho^{\gamma_{n}}-\eta_n^{2}=\]
\[(-\Delta)^{-1}\nabla \cdot (\nabla \cdot \vc{\tau_{1n}}).
\]
By following the same procedure as in \cite{DonatelliTrivisa2017} and by taking the limit as $n\rightarrow \infty$ we end up with relation,
\begin{equation} \nonumber
 \begin{split}
 & 2\Big[\overline{s}_{t}+\nabla \cdot(\vu \overline{s})\Big]+\overline{\rho^{\gamma+1}} \\
 &= -\rho\eta^{2}-\rho\Big[(-\Delta)^{-1}\nabla \cdot(\nabla \cdot \sigma-\nabla \eta)\Big] +\frac{d}{dt}\Big[\rho (-\Delta)^{-1}\nabla \cdot(\rho \vu)\Big] \\
 &+ \nabla \cdot\Big[\rho \vu (-\Delta)^{-1}\nabla \cdot(\rho \vu)\Big] \\
 &+ \rho \Big[(-\Delta)^{-1}\partial_{i}\partial_{j}(\rho \vu_{i}u_{j}) -u\cdot \nabla (-\Delta)^{-1}\nabla \cdot(\rho \vu)\Big].
 \end{split}
\end{equation}

Let $s=\rho \log \rho$, exactly as before (for the details of the proof we refer to \cite{DonatelliTrivisa2017}), we obtain that
\begin{equation} \nonumber
 \begin{split}
 & 2\Big[{s}_{t}+\nabla \cdot(\vu {s})\Big]+ \rho \overline{\rho^{\gamma}} \\
 &=-\rho\eta^{2} -\rho\Big[(-\Delta)^{-1}\nabla \cdot(\nabla \cdot \sigma-\nabla \eta)\Big] +\frac{d}{dt}\Big[\rho (-\Delta)^{-1}\nabla \cdot(\rho \vu)\Big] \\
 &+ \nabla \cdot\Big[\rho \vu (-\Delta)^{-1}\nabla \cdot(\rho \vu)\Big] \\
 &+ \rho \Big[(-\Delta)^{-1}\partial_{i}\partial_{j}(\rho \vu_{i}\vu_{j}) -u\cdot \nabla (-\Delta)^{-1}\nabla \cdot(\rho \vu)\Big].
 \end{split}
\end{equation}
Comparing the last two relations, we have
\begin{equation}
\partial_{t}(\bar{s}-s)+\Div\left ((\bar{s}-s)\vu_{n}\right)=-\overline{\rho\Div\vu}+\rho_{n}\Div\vu_{n}
\label{e13}
\end{equation}
and
\begin{equation}
\partial_{t}(\bar{s}-s)+\Div\left ((\bar{s}-s)\vu_{n}\right)=\frac{1}{2}\left(\rho\pi-\overline{(\rho_{n})^{\gamma_{n}+1}}\right).
\label{e14}
\end{equation}
Now, using that
$$(\rho)^{\gamma_{n}}\rightarrow \mathbf{1}_{\{\rho=1\}}, \quad \text{a.e. in $L^{p}((0,T)\times \Omega)$},$$
which yields
$$(\rho)^{\gamma_{n}}(\rho_{n}-\rho)\rightharpoonup 0,$$
we obtain
\begin{equation}
\overline{(\rho_{n})^{\gamma_{n}+1}}-\rho\overline{(\rho_{n})^{\gamma_{n}}}=\overline{(\rho_{n})^{\gamma_{n}}(\rho_{n}-\rho)}=\overline{((\rho_{n})^{\gamma_{n}}-\rho^{\gamma_{n}})(\rho_{n}-\rho)}\geq 0.
\label{e15}
\end{equation}
From \eqref{e15} we obtain, 
\begin{equation}
\rho \pi=\rho\overline{(\rho_{n})^{\gamma_{n}}}\leq \overline{(\rho_{n})^{\gamma_{n}+1}}.\nonumber
\end{equation}
We integrate  \eqref{e14} in space to get
\begin{equation}
\partial_{t}\int_{\Omega}(\bar{s}-s)dx\leq 0. \nonumber
\end{equation}
Now, since $(\bar{s}-s)|_{t=0}=0$ and by the convexity of $s$ we have $s\leq \bar{s}$ and $s=\bar{s}$.
Therefore, from \eqref{e14} we obtain
\begin{equation}
\rho \pi=\overline{(\rho_{n})^{\gamma_{n}+1}}
\label{e18}
\end{equation}
Moreover we have 
\begin{equation}
(\rho_{n})^{\gamma_{n}+1}\geq (\rho_{n})^{\gamma_{n}}-\varepsilon.
\label{e19}
\end{equation}
indeed it is sufficient to use the  property $x^{\gamma_{n}+1}\geq x^{\gamma_{n}}-\varepsilon$, $\varepsilon>0$ and any $x\geq 0$ in the case $x=\rho$.
By using \eqref{e18} and by passing to the weak limit in \eqref{e19}  we end up with
$$
\rho\pi\geq \pi-\varepsilon,
$$
and, as $\varepsilon\to 0$ we conclude with
\begin{equation}
\rho\pi\geq \pi.
\label{e20}
\end{equation}
The last issue to be proved  is  $\rho\pi\leq \pi$. Since $\rho\pi$ is not defined almost everywhere in order to give a meaning to the inequality we want to prove we define by $\omega_{k}$ a smoothing sequence in the space and time variables  as follows
$$\omega_{k}=k^{4}\omega(k\cdot),$$
$$\omega\in C^{\infty}(\R^{4}),\  \omega\geq 0, \quad \int_{\R^{4}}\!\!\!\!\!\!\!\!\!-\omega dxdt=1,\quad spt\omega\in B_{1}(\R^{4}).$$
We denote by $\rho_{k}$ and $\pi_{k}$ a sequence of smooth functions defined as
$$\rho_{k}=\rho\ast\omega, \qquad \pi_{k}=\pi\ast\omega$$
and we have that
$$\rho_{k}\rightarrow \rho \quad \text{in $C([0,T];L^{p})\cap C([0,T];H^{-1})$},$$
$$\pi_{k}\rightarrow \pi\quad \text{in $W^{-1,2}\cap L^{1}(L^{q})$},$$
for any $p,q$ such that $1/p +1/q=1$.
Hence we can rewrite $(\rho-1)\pi$ as
\begin{equation}
(\rho-1)\pi=(\rho_{k}-1)\pi_{k}+(\rho-\rho_{k})\pi_{k}+(\rho-1)(\pi-\pi_{k})
\label{e21}
\end{equation}
Since $\rho_{k}\leq 1$, as   $k\to \infty$ in \eqref{e21} we obtain
\begin{equation}
\rho\pi-\pi\leq 0
\label{e22}
\end{equation}
Considering together and \eqref{e20}  and \eqref{e22} we have \eqref{e12} and  we conclude the proof of the Theorem \ref{MT2}.

\subsection{Proof of the Theorem \ref{MT}}

The proof of the  Theorem \ref{MT} is a consequence of the Theorem \ref{MT2}. 
The only thing we have to check is that the condition \eqref{1.18} holds in the sense of distribution. This last issue is a consequence of the following lemma (for the proof we refer to \cite{LionsMasmoudi-1999}, Lemma 2.1).
\begin{lemma}
Let $\vu\in L^{2}(0,T;H^{1}_{loc}(\Omega))$ and $\rho\in L^{2}_{loc}((0,T)\times \Omega)$ satisfying 
 $$\partial_t \rho_n +\Div(\rho_n \vu_n)=0, \quad \text{in $(0,T)\times \Omega$},$$
 $$\rho(0)=\rho_{0},$$
 then the following two assertion are equivalent
 \begin{itemize}
 \item[(i)] $\Div \vu=0$, a.e. on $\{ \rho \geq 1\}$ and $0\leq \rho_{0}\leq 1$.
  \item[(ii)] $0\leq \rho\leq 1$.
 \end{itemize}
 \end{lemma}

The final step is to obtain  the energy inequality \eqref{energy} that we require our global weak solutions have to satisfy.
By applying  the convergence results proved in the Theorem \ref{MT2} we can pass into the weak limit in the energy inequality \eqref{energyapprox} and we get
\[
 \int_{\Omega} \left[\frac{1}{2} \rho |\vu|^2 +\xi \eta^{2} + k(l-(K+1)) {\mathcal   F}(\eta) + k \int_D M {\mathcal F}\left(\frac{\psi}{M}\right) d \vc{q}\right] dx 
\]
\[
+\mu^S\int_{0}^{t}\!\! \int_{\Omega} |D(\vu) - \frac{1}{d} (\Div_x \vu) \mathbb{I}|^2 dx + \mu^B \int_{0}^{t}\!\!\int_{\Omega} |\Div_x \vu|^2 dx 
\]
\[
+ 2 \e \xi \int_{0}^{t}\!\!\int_{\Omega} |\Grad_x \eta|^2 dx + 4 \e k (L - (K+1))\int_{0}^{t}\!\! \int_{\Omega} |\nabla_x \sqrt{\eta}|^2 dx
\]
\[
+ 4k \e\int_{0}^{t}\!\! \int_{\Omega \times D} M |\Grad_x \sqrt{\frac{\psi}{M}}|^2 d \vc{q} dx 
\]
\[
+ \frac{k}{\lambda} \sum_{i=1}^K \sum_{j=1}^K\int_{0}^{t}\!\! \int_{\Omega \times D} M \nabla_{\vc{q}_j} \sqrt{\frac{\psi}{M}} \nabla_{\vc{q}_i}  \sqrt{\frac{\psi}{M}} d \vc{q} dx\]
\[\leq  
\frac{1}{2}  \int_{\Omega}\rho_0|\vc{u_0}|^2 dx + k \int_{\Omega \times D} M \mathcal{F}(\widehat{\psi}_0) d \vc{q} dx+\liminf_{n\to \infty}\int_{\Omega}dx\frac{(\rho_{n0})^{\gamma_{n}}}{\gamma_{n}}+\int_{\Omega} \rho f \cdot \vc{u} dx,
\]
a.e in $t$. Now, if we take, for any $n>2$, $\rho_{n0}=\rho_{0}$, $m_{n0}=m_{0}$ and since $0\leq\rho_{0}\leq 1$, then
$$\liminf_{n\to \infty}\int_{\Omega}\frac{(\rho_{n0})^{\gamma_{n}}}{\gamma_{n}}dx=0$$
and we end up with the energy inequality \eqref{energy}.

\section{Related models} \label{S5}
We conclude this paper by mentioning two related models for the problem $\mathbf{(P_F)}$, where the conditions \eqref{1.19}-\eqref{1.20} can be generalized.
\subsection{General pressure law fluid}
The free-boundary conditions \eqref{1.19}-\eqref{1.20} can be extended to include the case with a general fluid pressure, namely 
\begin{equation}
\label{5.1}
p_{F} \ge  p(1) \quad \mbox{a.e} \,\, \mbox{in} \,\, \{ \rho = 1\}
\end{equation}
\begin{equation}
\label{5.2}
p_{F} =  p(\rho)  \quad \mbox{a.e} \,\, \mbox{in} \,\, \{ \rho < 1\}
\end{equation}
The polymer behaviour in this case is that of a barotropic fluid in the region $ \{ \rho < 1\}$ and the condition  \eqref{cp} becomes
$$ \rho (p_F - p(\rho)) = p_F - p(\rho).$$
This generalization requires only some technical changes in the energy estimates which can be treated in a similar manner.

\subsection{Congestion constraints}
Our \ analysis  \ can \  accommodate non-homogeneous congestions constraints, i.e. a non homogeneous threshold for the pressure. In this case \eqref{1.19}-\eqref{1.20} have the form
 \begin{equation}
 \label{5.3}
p_F  \ge  0  \quad \mbox{a.e} \,\, \mbox{in} \,\, \{ \rho = \rho^{*}(x)\}
\end{equation}
\begin{equation}
\label{5.4}
p_{F} =  0   \quad \mbox{a.e} \,\, \mbox{in} \,\, \{ \rho < \rho^{*}(x)\}.
\end{equation}
This can be achieved by introducing in the approximating system \eqref{3.1}-\eqref{3.4} an approximating pressure of the form
$${p_F}_n = \left(\frac{\rho_n}{\rho^{*}}\right)^{\gamma_n}.$$

\section{Acknowlegments}
 The research of D.D. leading to these results was supported  by the European Union's Horizon 2020 Research and Innovation Programme under the Marie Sklodowska-Curie Grant Agreement No 642768 (Project Name: ModCompShock). K.T.  gratefully acknowledges the support  in part by the National Science Foundation under the grant DMS-1614964  and by the Simons Foundation under the Simons Fellows in Mathematics Award 267399. Part of this research was performed during the visit of K.T. at University of L'Aquila which was supported by the European Union's Horizon 2020 Research and Innovation Programme under the Marie Sklodowska-Curie Grant Agreement No 642768 (Project Name: ModCompShock).


\begin{thebibliography}{00}

\bibitem{BaeTrivisa2012} H. Bae and K. Trivisa, On the Doi model for the suspensions of rod-like molecules in compressible fluids, {\it Math. Models Methods Appl. Sci.}  {\bf 22} (2012),  39 pp.

\bibitem{BT2013} H. Bae and K. Trivisa, On the Doi model for the suspensions of rod-like molecules: Global-in-time existence,  {\it Commun.\ Math.\ Sci.} {\bf 11} (2013), 831�850.

\bibitem{BS2008} J. W. Barrett and E. S\"{u}li, Existence of global weak solutions to dumbbell models for dilute polymers with microscopic cut-off, {\it Math. Models Methods Appl. Sci.} {\bf 18} (2008), 935�971.

\bibitem{BS2011} J. W. Barrett and E. S\"{u}li, Existence and equilibration of global weak solutions to kinetic models for dilute polymers I: Finitely extensible nonlinear bead-spring chains, {\it Math. Models Methods Appl. Sci.} {\bf 21} (2011), 1211�1289.

\bibitem{BS2012A} J. W. Barrett and E. S\"{u}li, Existence and equilibration of global weak solutions to kinetic models for dilute polymers II: Hookean-type bead-spring chains, {\it Math. Models Methods Appl. Sci.} {\bf 22} (2012), 84 pp.

\bibitem{BS2012B} J. W. Barrett and E. S\"{u}li, Existence of global weak solutions to finitely extensible nonlinear bead-spring chain models for dilute polymers with variable density and viscosity, {\it J. Differential Equations} {\bf 253} (2012), 3610�3677.

\bibitem{BS2012C} J. W. Barrett and E. S\"{u}li, Finite element approximation of finitely extensible nonlinear elastic dumbbell models for dilute polymers, {\it M2AN Math. Model. Numer. Anal.} {\bf 46} (2012) 949�978.

\bibitem{BS2012D} J. W. Barrett and E. S\"{u}li, Reflections on Dubinsk?�i�s nonlinear compact embedding theorem, {\it Publ. Inst. Math. (Belgrade)} {\bf 91} (2012) 95�110.

\bibitem{BS2016} J. W. Barrett and E. S\"{u}li, Existence of global weak solutions to compressible isentropic finitely extensible nonlinear bead-spring chain models for dilute polymers, {\it  Math. Models Methods Appl. Sci.} {\bf 26} (2016), {\bf no. 3}, 469�568.

\bibitem{BS2018} J. W. Barrett and E. S\"{u}li,  Existence of global weak solutions to the kinetic Hookean dumbbell model for incompressible dilute polymeric fluids. {\it Nonlinear Anal. Real World Appl.} {\bf 39} (2018), 362�395.

\bibitem{BAH1977} R.B.\ Bird, R.\ Amstrong, O.\ Hassager, {\em Dynamics of Polymeric Liquids}, {\bf vol. 1}. {\em Wiley, New York} (1977)

\bibitem{BAH1987} R.B.\ Bird, R.\ Amstrong, O.\ Hassager,  {\em Dynamics of Polymeric Liquids. Kinetic Theory,} {\bf vol. 2}. {\em Wiley, New York} (1987)

\bibitem{BCAH1987} R.\ Bird, C.\ Curtiss, R.\ Armstrong, and O.\ Hassager, {\em Dynamics of Polymeric Liquids,} {\bf Vol.2}, Kinetic Theory, {\em John Wiley \& Sons, NewYork,}1987.

\bibitem{Constantin2007} P. Constantin, C. Fefferman, E.S. Titi and  A. Zarnescu, Regularity of coupled two-dimensional nonlinear Fokker-Planck and Navier-Stokes systems, {\it Comm. Math. Phys.}, {\bf 270} (2007), no.3, 789--811.


\bibitem{Constantin2008} P. Constantin and N. Masmoudi, Global well-posedness for a Smoluchowski equation coupled with Navier-Stokes equations in 2D, {\it Comm. Math. Phys.}, {\bf 278} (2008), no.1, 179--191.

\bibitem{Dain2006} S. Dain, Generalised Korn�s inequality and conformal Killing vectors, {\it Calc. Var. Partial Differential Equations} {\bf 25} (2006) 535�540.

\bibitem{DegondLemou2002} P.\ Degond, M.\ Lemou, M.\ Picasso, Viscoelastic fluid models derived from kinetic equations for polymers. {\em SIAM J. Appl. Math.} {\bf 62} (2002) {\bf (5)}, 1501--1519.

\bibitem{DoiEdwards1986} M. Doi and  S.F. Edwards, The theory of polymer dynamics, Oxford University press, 1986.

\bibitem{DonatelliTrivisa2017} D. Donatelli and K. Trivisa, On a free boundary problem for polymeric fluids: global existence of weak solutions. {\it NoDEA Nonlinear Differential Equations Appl.} {\bf 24} (2017), {\bf no. 5}, Art. 51, 20 pp.

\bibitem{Dubinski�19656} J. A. Dubinski$\breve{i}$, Weak convergence for nonlinear elliptic and parabolic equations, {\it Mat. Sb. (N.S.)} {\bf 67} (1965) 609--642.

\bibitem{Feireisl2001} E. Feireisl, On compactness of solutions to the compressible isentropic Navier-Stokes equations when the density is not square integrable, {\it Comment. Math. Univ. Carolin.}, {\bf 42} (2001), {\bf no.1}, 83--98.

\bibitem{Feireisl2003} E. Feireisl, Dynamics of viscous compressible fluids, Oxford University Press, 2003.


\bibitem{Lions1998} P.L. Lions, Mathematical topics in fluid dynamics, v2, Compressible models, Oxford University Press, 1998.

\bibitem{LionsMasmoudi-1999} P.L.\ Lions and N.\ Masmoudi, On a free boundary barotropic model, {\it Ann. \ Inst. Henri Poincar\'{e}}, {\bf 16} (1999) {no. 3}, 373-410.
\bibitem{Ottinger1996} H.C.\  \"{O}ttinger, {\em Stochastic Processes in Polymeric Fluids}. {\em Springer, Berlin} (1996). Tools and examples for developing simulation algorithms.
\bibitem{OwensPhillips2002} R.G.Owens, T.N.\ Phillips, {\em Computational Rheology}. Imperial College Press, London (2002)

\end{thebibliography}
\end{document}